\documentclass{amsart}
\usepackage{amsmath,amssymb,amsthm}
\usepackage{mathrsfs}
\usepackage{stmaryrd}
\usepackage[colorlinks=true,hyperindex, linkcolor=magenta, pagebackref=false, citecolor=cyan]{hyperref}
\usepackage[arrow, matrix, curve]{xy}

\newtheorem{lem}{Lemma}[section]
\newtheorem{cor}[lem]{Corollary}
\newtheorem{prop}[lem]{Proposition}
\newtheorem{thm}[lem]{Theorem}
\newtheorem{def-prop}[lem]{Definition-Proposition}
\newtheorem{def-lem}[lem]{Definition-Lemma}
\theoremstyle{definition}
\newtheorem{definition}[lem]{Definition}
\newtheorem{example}[lem]{Example}

\theoremstyle{definition}
\newtheorem{remark}[lem]{Remark}
\newtheorem{notation}[lem]{Notation}

\newenvironment{mat}
{\left(\begin{smallmatrix}}
	{\end{smallmatrix}\right)}

\newcommand{\Hom}{\operatorname{Hom}}
\newcommand{\End}{\operatorname{End}}
\newcommand{\Aut}{\operatorname{Aut}}
\newcommand{\Gal}{\operatorname{Gal}}

\newcommand{\Ind}{\operatorname{Ind}}
\newcommand{\cInd}{\operatorname{c-Ind}}
\newcommand{\GL}{\operatorname{GL}}
\newcommand{\bQ}{\mathbf{Q}}
\newcommand{\bZ}{\mathbf{Z}}
\newcommand{\bF}{\mathbf{F}}
\newcommand{\Fpbar}{\overline{\mathbf{F}}_p}

\newcommand{\wt}{\widetilde}
\newcommand{\ol}{\overline}

\newcommand{\adm}{\operatorname{adm}}

\newcommand{\Mod}{\operatorname{Mod}}
\newcommand{\sm}{\operatorname{sm}}

\newcommand{\sq}{\operatorname{sq}}

\newcommand{\Sym}{\operatorname{Sym}}

\newcommand{\fadm}{\operatorname{fg.adm}}
\newcommand{\GQp}{\mathscr{G}_{\bQ_p}}
\newcommand{\fin}{\operatorname{fin}}
\newcommand{\phigamma}{\operatorname{\Phi\Gamma^{\operatorname{\acute{e}t}}_{k((X))}}}
\newcommand{\mphigamma}{\operatorname{\wt{\Phi\Gamma}^{\operatorname{\acute{e}t}}_{k((X)),\iota}}}

\newcommand{\bOmega}{\mathbf{\Omega}}

\newcommand\noloc{%
	\nobreak
	\mspace{6mu plus 1mu}
	{:}
	\nonscript\mkern-\thinmuskip
	\mathpunct{}
	\mspace{2mu}
}

\DeclareFontShape{OMX}{cmex}{m}{n}{
	<-7.5> cmex7
	<7.5-8.5> cmex8
	<8.5-9.5> cmex9
	<9.5-> cmex10
}{}
\SetSymbolFont{largesymbols}{normal}{OMX}{cmex}{m}{n}
\SetSymbolFont{largesymbols}{bold}  {OMX}{cmex}{m}{n}

\begin{document}

\title{A mod-$p$ metaplectic Montr\'{e}al functor}

\author{Robin Witthaus}
\email{robin.witthaus@stud.uni-due.de}
\date{\today}

\begin{abstract}
		We extend Colmez's functor defined for $\GL_2(\bQ_p)$ to the category of finitely generated smooth admissible mod-$p$ representations of the two-fold metaplectic cover of $\GL_2(\bQ_p)$. We compute the images of the absolutely irreducible genuine objects and obtain a bijection between the genuine supersingular representations and four-dimensional irreducible Galois representations invariant under twist by all characters of order two. Restricted to genuine objects, the extended functor naturally takes values in the category of what we call metaplectic Galois representations -- Galois representations with a certain extra structure encoding the aforementioned  twist-invariance.
\end{abstract}

\maketitle
\tableofcontents

\section{Introduction}
\label{sec:introduction}
\subsection{The $p$-adic and mod-$p$ local Langlands correspondence for $\GL_2(\bQ_p)$} Integral weight modular forms may be viewed as automorphic representations of $\GL_2(\mathbf{A}_{\mathbf{Q}})$ and, after fixing an auxiliary prime $p$, they have a Galois representation of $\Gal(\ol{\bQ}/\bQ)$ with coefficients in an algebraic closure of $\bQ_p$ attached to them satisfying local-global compatibility at all finite primes; that is to say, for each finite place $\ell$, the smooth complex representation of $\GL_2(\bQ_{\ell})$ appearing as the local component of the automorphic representation is determined by the restriction of the Galois representation to the decomposition group $\Gal(\ol{\bQ}_{\ell}/\bQ_{\ell})$ at $\ell$ through the local Langlands correspondence after passage to the corresponding Weil-Deligne representation. If $\ell\neq p$, then the local Galois representation at $\ell$ can essentially be recovered from the corresponding representation of $\GL_2(\bQ_{\ell})$. At $\ell=p$ however, passage to the Weil-Deligne representation is not fully faithful and hence the global correspondence does not restrict to a bijective correspondence locally at $p$. As explained in \cite{Breuil_intro}, the search for a $p$-adic local Langlands correspondence grew out of the desire to fix this issue. It roughly speaking aims to relate irreducible admissible unitary Banach space representation of $\GL_2(\bQ_p)$ defined over a finite extension of $\bQ_p$ to Galois representations of $\Gal(\ol{\bQ}_p/\bQ_p)$ defined over the same finite extension.

The starting point of the $p$-adic and mod-$p$ local Langlands correspondence for $\GL_2(\bQ_p)$ was the classification of the smooth irreducible (admissible) $\Fpbar$-linear representation by Barthel-Livn\'{e} \cite{Barthel-Livne} (working more generally for finite extensions of $\bQ_p$) and Breuil \cite{Breuil_I}. Namely, Barthel and Livn\'{e} were able to describe all the irreducible representations occuring as a subquotient of a parabolic induction, and, building on their approach via universal spherical Hecke modules, Breuil managed to parametrize all the other --so-called \textit{supersingular}-- representations. This parametrization allowed him to write down a bijection between the set of isomorphism classes of supersingular representations and the set of isomorphism classes of two-dimensional continuous irreducible $\Fpbar$-linear representations of $\GQp$, which already suggested that there should be a mod-$p$ local Langlands correspondence. Motivated by the possible existence of a $p$-adic local Langlands correspondence, Breuil \cite{Breuil_II} computed the reduction of (the completion of) locally algebraic unramified principal series depending on small weights, which led him to the formulation of the semi-simple mod-$p$ local Langlands correspondence as it can be found at the beginning of \textit{loc.\ cit.} In \cite{colmez1}, Colmez observes a close relation between $(\varphi,\Gamma)$-modules and smooth representations of the mirabolic subgroup of $\GL_2(\bQ_p)$ (and thus of the standard Borel if a central character is fixed). Inspired by the work \cite{colmez3} of Colmez, Berger and Breuil \cite{berger_breuil} compute the unitary completion of locally algebraic principal series representations and relate it to two-dimensional crystalline Galois representations via $(\varphi,\Gamma)$-modules; they also prove compatibility with Breuil's semi-simple mod-$p$ correspondence. Colmez \cite{colmez2} has turned his previously mentioned observation into an exact functor from the category of finite length $\GL_2(\bQ_p)$-representations (admitting a central character) to the category of $(\varphi,\Gamma)$-modules, which is equivalent to the category of Galois representations of $\GQp$ via Fontaine's equivalence. This functor is not only defined for mod-$p$ coefficients, in which case it recovers Breuil's correspondence after semi-simplification, but also for torsion coefficients and thus for Banach space representations yielding a \textit{functorial} approach to a $p$-adic Langlands correspondence. This functor has been studied in more detail ever since, see for example the works \cite{cdp} and \cite{paskunas} of Colmez-Dospinescu-Pa\v{s}k\={u}nas and Pa\v{s}k\={u}nas, respectively.

\subsection{Motivation and goal} Half-integeral weight modular forms, on the other hand, may be viewed as genuine automorphic representations of the so-called \textit{metaplectic cover $\wt{\GL}_2(\mathbf{A}_{\mathbf{Q}})$} of $\GL_2(\mathbf{A}_{\mathbf{Q}})$ -- a central extension by the roots of unity $\mu_2$ in $\bQ$, where ``genuine'' means that $\mu_2$ acts non-trivially. It is natural to ask whether these analytic objects also have a Galois representation attached to them in a reasonable way. One possible answer to this question is to make use of Shimura's correspondence (\cite{shimura}, \cite{Gelbart_and_co}, \cite{flicker}), which assigns an integral weight modular form to one of half-integral weight in a suitable Hecke-equivariant fashion, and then pass to the corresponding global Galois representation. This is unsatisfying as it does not give any new information about Galois representations, so it seems worthwhile to look for a direct way of realizing such a metaplectic correspondence.

We would like to investigate this question from a $p$-adic point of view locally at $p$, i.e.\ letting $\wt{\GL}_2(\bQ_p)$ denote the pullback of the global covering to the local group at $p$, we are interested in a metaplectic version of the $p$-adic local Langlands correspondence relating irreducible admissible unitary genuine $p$-adic Banach space representations of $\wt{\GL}_2(\bQ_p)$ to certain $p$-adic Galois representations of $\GQp$. Both, $p$-adic Banach space representations residually of finite length and $p$-adic Galois representations, can be reduced mod-$p$ (and semi-simplified to obtain something well-defined), so any $p$-adic correspondence will give rise to a semi-simple mod-$p$ metaplectic correspondence (at least if $p>2$ so that there is a non-trivial character $\mu_2\to \bF_p^{\times}$). In fact, it seems easier to first study the mod-$p$ situation. A complete classification of the smooth irreducible genuine $\Fpbar$-representations of $\wt{\GL}_2(\bQ_p)$ is given in \cite{witthaus_meta} and the goal of this paper is to establish a relation to mod-$p$ Galois representations by extending Colmez's functor.

\subsection{Results} We fix an odd prime number $p$ and a finite field $k$ of characteristic $p$. We let $G=\GL_2(\bQ_p)$. The metaplectic covering
\[
1\to \mu_2=\{\pm 1\}\to \wt{G}\to G\to 1
\]
is described by an explicit $2$-cocycle relying on the Hilbert symbol, and it splits over the maximal compact open subgroup. Let $\iota\colon \mu_2\hookrightarrow k^{\times}$ be the non-trivial group homomorphism. We let $\Mod^{\sm}_{\wt{G}}(k)$ denote the category of smooth $k$-linear representations of $\wt{G}$. Of particular interest for us is the full subcategory $\Mod^{\sm}_{\wt{G},\iota}(k)$ of \textit{genuine} representations, i.e.\ those on which the central subgroup $\mu_2$ acts via $\iota$.

The category $\phigamma$ of \'{e}tale $(\varphi,\Gamma)$-modules over $k((X))$ is equivalent to the category $\Mod^{\fin}_{\GQp}(k)$ of finite dimensional smooth $k$-linear representations of $\GQp$ via Fontaine's equivalence. Colmez's functor $\mathbf{D}\colon \Mod^{\fadm}_{G}(k)\to \phigamma$ is defined on the category of finitely generated admissible $k$-representations of $G$ and has the property that the image of an object only depends on the restriction of that object to the mirabolic subgroup $P=\begin{mat}
	\bQ_p^{\times} & \bQ_p\\ 0 & 1
\end{mat}$. More formally, using Breuil's generalization \cite{Breuil_phi}, one may view Colmez's functor as defined on the category of smooth $P$-representations at the cost of replacing the target category by its ind-category: $\mathbf{D}\colon \Mod^{\sm}_P(k)\to \Ind(\phigamma)$. When composed with the restriction functor $\Mod^{\fadm}_G(k)\to \Mod^{\sm}_P(k)$, this more general functor takes values in the subcategory of compact objects and recovers Colmez's functor. 

What allows us to extend Colmez's functor is the observation that the pullback $\wt{P}$ of the metaplectic covering to $P$ is simply equal to the product of groups $P\times \mu_2$. We may therefore consider the restriction functor $\Mod^{\fadm}_{\wt{G}}(k)\to \Mod^{\sm}_P(k)$ and Corollary \ref{cor_compact} establishes:

\begin{prop}
	The composition $\Mod^{\fadm}_{\wt{G}}(k)\to \Mod^{\sm}_P(k)\xrightarrow{\mathbf{D}} \Ind(\phigamma)$ takes values in the subcategory of compact objects, yielding the exact functor
	\[
	\mathbf{D}\colon \Mod^{\fadm}_{\wt{G}}(k)\to \phigamma,
	\]
	whose restriction to the full subcategory $\Mod^{\fadm}_G(k)$ recovers Colmez's functor.
\end{prop}

Putting $P^+=\begin{mat}
	\bZ_p\setminus \{0\} & \bZ_p\\ 0 & 1
\end{mat}$ and letting $Z$ denote the center of $G$, the key for establishing this result is to show the following: Given a finite dimensional $\wt{K}\wt{Z}$-stable subspace $M$ of $\pi\in \Mod^{\fadm}_{\wt{G}}(k)$ generating $\pi$ as a $G$-representation, the $P^+$-subrepresentation of $\pi$ generated by $M$ is admissible in the sense that its subspace of $\begin{mat}
	1 & \bZ_p\\ 0 & 1
\end{mat}$-invariants is finite dimensional. This can either be proved by making sure that an appropriate version of Emerton's formalism of ordinary parts also works for $\Mod^{\fadm}_{\wt{G},\iota}(k)$ to establish an analogue of \cite[Theorem 3.4.7]{ord_parts_II}, or it can be deduced from known properties of the irreducible objects in $\Mod^{\fadm}_{\wt{G}}(k)$. We choose to take the latter approach.\\

Denote by $\mathscr{V}\colon \phigamma\simeq \Mod^{\fin}_{\GQp}(k)$ Fontaine's equivalence and put $\mathbf{V}=\mathscr{V}\circ \mathbf{D}$. Evaluating the extended functor at the genuine supersingular representations, we obtain our main result (Theorem \ref{ss_bij_galois}):

\begin{thm}
	The functor $\mathbf{V}\colon \Mod^{\fadm}_{\wt{G}}(k)\to \Mod^{\fin}_{\mathscr{G}_{\bQ_p}}(k)$ induces a bijection
	\[
	\left\{\begin{array}{c}
		\text{Absolutely irreducible}\\
		\text{genuine supersingular}\\
		\text{$k$-representations of $\wt{G}$}
	\end{array}\right\} \cong \left\{\begin{array}{c}
		\text{Absolutely irreducible}\\
		\text{smooth $\rho\colon \GQp\to \GL_4(k)$}\\
		\text{s.t.\ $\rho\cong \rho\otimes \chi$}\\
		\text{for all characters $\chi$ with $\chi^2=\mathbf{1}$}
	\end{array}\right\},
	\]
	both sides considered up to isomorphism.
\end{thm}

This theorem consists of two parts, namely surjectivity and injectivity. It is not difficult to write down what the Galois representations appearing as the target of the bijection look like. The hard part for showing surjectivity is to compute the image of a genuine supersingular representation explicitly, see Proposition \ref{image_ss} for a description. The detailed study of the supersingular objects in \cite{witthaus_meta} allows us to put ourselves into a setup to which we can apply a general result of Florian Herzig on computing Galois representations attached to certain ``nice'' smooth $k$-representations of $P^+$, see Section \ref{section_Galois_calculate}. Injectivity comes down to the fact, proved in \cite{witthaus_meta}, that two genuine supersingular representations are isomorphic if and only if their restrictions to $P$ are.

There is a conceptual explanation for the twist-invariance under all characters of order two. Namely, consider the pullback $1\to \mu_2\to \wt{\bQ_p^{\times}}\to \bQ_p^{\times}\to 1$ of the metaplectic cover to the center $Z\cong \bQ_p^{\times}$ of $G$, and for $\tilde{z}\in \wt{\bQ_p^{\times}}$ with image $z\in \bQ_p^{\times}$, define $\chi_{\tilde{z}}=(\mu_{-1}^{v(z)}\omega^{v(z)}\mu_{\omega(z)})^{\frac{p-1}{2}}$, where $v\colon \bQ_p^{\times}\to \bZ$ is the $p$-adic valuation, $\mu_{?}$ is the unramified character mapping $p$ to $?$ and $\omega$ is the mod-$p$ cyclotomic character. The conjugation action of an element $\tilde{z}\in \wt{Z}\cong \wt{\bQ_p^{\times}}$ is given by the $\wt{G}\supset \mu_2$-valued character $\chi_{\tilde{z}}\circ \det$.
In particular, the action of $\tilde{z}$ on a genuine representation $\pi$ defines a $\wt{G}$-equivariant isomorphism between $\pi$ and its twist by $\chi_{\tilde{z}}\circ \det$. Now, the functor $\mathbf{D}$ is compatible with twist by a character $\chi$ in the sense that $\mathbf{D}(\pi\otimes \chi\circ \det)=\mathbf{D}(\pi)\otimes \chi$, meaning that the underlying vector space remains the same but the $(\varphi,\Gamma)$-action is twisted according to $\chi$, see Definition \ref{phi_gamma_twist}. Thus, any Galois representation arising as the image of a genuine representation under $\mathbf{V}$ is invariant under twist by the characters $\chi_{\tilde{z}}$, for $\tilde{z}\in \wt{\bQ_p^{\times}}$, which assemble into a \textit{surjective} homomorphism $\wt{\bQ_p^{\times}}\twoheadrightarrow \Hom^{\operatorname{cts}}(\bQ_p^{\times},\mu_2)$. This more conceptual point of view leads to the following notion.

\begin{definition}
	A \textit{metaplectic $(\varphi,\Gamma)$-module over $k((X))$} is a pair $(D,\theta)$ consisting of $D\in \phigamma$ and $\theta\colon \wt{\bQ_p^{\times}}\to \Aut_{k((X))}(D)$ a smooth and genuine $k((X))$-linear action such that $\theta(\tilde{z})\colon D\cong D\otimes \chi_{\tilde{z}}$ is $(\varphi,\Gamma)$-equivariant. Denote the resulting category by $\mphigamma$.
\end{definition}

The extended functor then naturally induces the \textit{metaplectic Montr\'{e}al functor} 
\[
\wt{\mathbf{D}}\colon \Mod^{\fadm}_{\wt{G},\iota}(k)\to \mphigamma.
\]
Using Fontaine's equivalence, one may translate the additional structure $\theta$ on a metaplectic $(\varphi,\Gamma)$-module $(D,\theta)$ into additional structure on the Galois representation $\mathscr{V}(D)$, and call the resulting object a metaplectic Galois representation. We note however that the additional structure on such a metaplectic Galois representation is \textit{not} simply given by a (smooth and genuine) action on the underlying $k$-vector space defining appropriate isomorphisms, see Section \ref{section_transferring}.\\

If we let $S\subset \bQ_p^{\times}$ denote the subgroup of squares, then the characters $\chi_{\tilde{z}}$ are trivial for all $\tilde{z}\in \wt{S}$, hence we have a restriction functor $\mphigamma\to \Mod^{\sm}_{\wt{S},\iota}(\phigamma)\simeq \Mod^{\sm}_S(\phigamma)$, where the equivalence follows from the equality of groups $\wt{S}=S\times \mu_2$ and is given by $-\boxtimes \iota$. This functor admits an explicit left adjoint denoted by $\Ind^{\wt{\bQ_p^{\times}}}_{\wt{S}}(-\boxtimes \iota)$, see Section \ref{section_Frobenius}, nicely interacting with the functor $\wt{\mathbf{D}}$ as explained in Proposition \ref{D_comp_Ind}.\\

We further utilize Colmez's construction \cite{colmez1} from $(\varphi,\Gamma)$-modules to smooth representations of the mirabolic subgroup, to obtain a functor $\wt{\bOmega}\colon \mphigamma \to \Mod^{\sm}_{\wt{B},\iota}(k)$ from metaplectic $(\varphi,\Gamma)$-modules to smooth genuine representations of the metaplectic Borel. Moreover, there is a natural transformation $\wt{\bOmega}\circ \wt{\mathbf{D}}\Rightarrow (-)|_{\wt{B}}$ of functors $\Mod^{\fadm}_{\wt{G},\iota}(k)\to \Mod^{\sm}_{\wt{B},\iota}(k)$, see Proposition \ref{can_B_equiv_map}.\\

Finally, we compute the images of the absolutely irreducible representations in $\Mod^{\sm}_{\wt{G},\iota}(k)$. As we will recall in more detail in the main text below, these are either supersingular or principal series representations. Putting $B_S=\begin{mat}
	\bQ_p^{\times} & \bQ_p\\ 0 & S
\end{mat}$, we have $\wt{B}_S=B_S\times \mu_2$, and any absolutely irreducible genuine principal series is of the form $\tilde{\pi}(\chi_1,\chi_2)=\Ind^{\wt{G}}_{\wt{B}_S}(\chi_1\otimes \chi_2 \boxtimes \iota)$ for smooth characters $\chi_1$, $\chi_2$ of $\bQ_p^{\times}$ and $S$, respectively. The isomorphism class of $\tilde{\pi}(\chi_1,\chi_2)$ is uniquely determined by the pair $(\chi_1|_S,\chi_2|_S)$. Denoting by $\tilde{\pi}(\chi_1,\chi_2)^{\operatorname{bc}}\subset \tilde{\pi}(\chi_1,\chi_2)\cong \Ind^{\wt{G}}_{\wt{B}}\Ind^{\wt{T}}_{\wt{T}_S}(\chi_1\otimes \chi_2\boxtimes \iota)$ the $\wt{B}$-invariant subspace of functions whose support is contained in the preimage of the big cell $B\begin{mat}
	0 & 1\\1 & 0
\end{mat}B$ under the projection $\wt{G}\twoheadrightarrow G$, we obtain the following result (Proposition \ref{image_ps}).

\begin{prop}
	Let $\chi_1,\chi_2\colon \bQ_p^{\times}\to k^{\times}$ be smooth characters, and let $\mathscr{D}(\chi_2)$ be the $(\varphi,\Gamma)$-module corresponding to $\chi_2$ viewed as a character of $\GQp$ via local class field theory.
	\begin{enumerate}
		\item[{\rm (i)}] The object $\Ind^{\wt{\bQ_p^{\times}}}_{\wt{S}}((\chi_1\chi_2)|_S \boxtimes \mathscr{D}(\chi_2)\boxtimes \iota)$ is irreducible in $\mphigamma$.
		\item[{\rm (ii)}] There is an isomorphism
		\[
		\wt{\mathbf{D}}(\tilde{\pi}(\chi_1,\chi_2)) \cong \Ind^{\wt{\bQ_p^{\times}}}_{\wt{S}}((\chi_1\chi_2)|_S \boxtimes \mathscr{D}(\chi_2)\boxtimes \iota)
		\]
		in the category $\mphigamma$.
		\item[{\rm (iii)}] The canonical $\wt{B}$-equivariant map $\wt{\bOmega}\wt{\mathbf{D}}(\tilde{\pi}(\chi_1,\chi_2))\to \tilde{\pi}(\chi_1,\chi_2)$ has image $\tilde{\pi}(\chi_1,\chi_2)^{\operatorname{bc}}$ and induces a non-split short exact sequence
		\[
		0\to \Ind^{\wt{T}}_{\wt{T}_S}(\chi_2\omega \otimes (\chi_1 \omega^{-1})|_S\boxtimes \iota)\to \wt{\bOmega}\wt{\mathbf{D}}(\tilde{\pi}(\chi_1,\chi_2))\to \tilde{\pi}(\chi_1,\chi_2)^{\operatorname{bc}}\to 0.
		\]
	\end{enumerate}
\end{prop}

Interestingly, these statements can be essentially reduced to the corresponding known results for principal series of $\GL_2(\bQ_p)$.

Concerning absolutely irreducible genuine supersingular representations, we have already explained above that their image under $\mathbf{V}$, or equivalently $\mathbf{D}$, is explicitly described and absolutely irreducible. Parts (iii) and (iv) of Proposition \ref{image_ss} further establish:

\begin{prop} Let $\pi$ be an absolutely irreducible genuine supersingular $k$-representation  of $\wt{G}$, and let $\zeta\boxtimes \iota\colon S\times \mu_2=\wt{S}\cong Z(\wt{G})\to k^{\times}$ be its central character.
	\begin{enumerate}
		\item[{\rm (i)}] The canonical $\wt{B}$-equivariant map $\wt{\bOmega}\wt{\mathbf{D}}(\pi)\xrightarrow{\cong} \pi$ is an isomorphism.
		
		\item[{\rm (ii)}] The metaplectic $(\varphi,\Gamma)$-module $\wt{\mathbf{D}}(\pi)$ is the unique object in $\mphigamma$ whose restriction to $S$ is isomorphic to $\zeta\boxtimes \mathbf{D}(\pi)\in \Mod^{\sm}_S(\phigamma)$ such that the $\wt{B}$-action on $\wt{\bOmega}\wt{\mathbf{D}}(\pi)$ extends to a smooth $\wt{G}$-action.
	\end{enumerate}
\end{prop}

\subsection{Outlook} We hope that this work serves as the starting point of a possible metaplectic $p$-adic local Langlands correspondence for $\GL_2(\bQ_p)$. Building on the computations of mod-$p$ reductions of locally algebraic unramified genuine principal series appearing in \cite{witthaus_meta}, one may write down a speculative semi-simple mod-$p$ metaplectic local Langlands correspondence relating the mod-$p$ automorphic side to metaplectic Galois representations of dimension $2n^2=8$, where $n^2=[L:\bQ_p]$ with $L$ being the compositum of all degree-two extensions of $\bQ_p$. However, real evidence for such a correspondence can only come from the $p$-adic picture. This makes a closer investigation of the relation between the completions of the aforementioned locally algebraic representations and certain (crystalline) Galois representations, as in the work of Breuil-Berger \cite{berger_breuil} for $\GL_2(\bQ_p)$, inevitable.\\

\noindent \textbf{Acknowledgements.}
This work is the second part of the author's PhD thesis and he is grateful to his advisor Vytautas Pa\v{s}k\={u}nas for suggesting the topic at hand and for helpful discussions during the past years. The results in Section \ref{section_Galois_calculate} are an elaborated version of parts of Florian Herzig's lectures \cite{herzig_notes} at the Spring School on the mod-$p$ Langlands Correspondence, which took place in Essen (online) in April 2021, and the author thanks him for outlining a proof of Proposition \ref{herzig_prop} on request and referring to \cite[Proposition 3.2.4.2]{herzig_co}, which simplified and conceptualized a previous version of the relevant part in this
paper. This work was funded by the DFG Graduiertenkolleg 2553.

\subsection{Conventions and notations} Throughout this work, we fix an odd prime number $p$ and a finite field $k$ of characteristic $p$ with an algebraic closure $\bar{k}$. We put $G=\GL_2(\bQ_p)$ and use the standard notations $B\supset T\supset Z$ to denote the subgroup of upper triangular matrices, the subgroup of diagonal matrices and the center of $G$, respectively, and we let $K=\GL_2(\bZ_p)$ be the maximal compact open subgroup. We identify $Z\cong \bQ_p^{\times}$ via the diagonal embedding and let $S\subset \bQ_p^{\times}$ denote the subgroup of squares.

Given a topological group $\mathcal{G}$, the category of smooth $k$-linear representation of $\mathcal{G}$ will be denoted by $\Mod^{\sm}_{\mathcal{G}}(k)$. Here, smooth means that the stabilizer of each vector is open. 

For a closed subgroup $\mathcal{H}\subset \mathcal{G}$, we let $\Ind^{\mathcal{G}}_{\mathcal{H}}(-)\colon \Mod^{\sm}_{\mathcal{H}}(k)\to \Mod^{\sm}_{\mathcal{G}}(k)$ be the smooth induction defining a left adjoint of the restriction functor $(-)|_{\mathcal{H}}$. If $\mathcal{H}$ is in fact open, we similarly denote by $\cInd^{\mathcal{G}}_{\mathcal{H}}(-)$ the compact induction defining a right adjoint of the restriction functor.

Finally, if $\mathcal{G}$ is abelian, we let $\mathcal{G}^{\sq}$ be its subgroup of squares.

	\section{Genuine representation theory}
	
In this section, we remind the reader of the classification of the smooth absolutely irreducible genuine mod-$p$ representations of the two-fold metaplectic cover of $G$ in terms of universal spherical Hecke modules, and then prove a finiteness result, which will play an important role for extending Colmez's functor.
	
\subsection{Recollection}\label{section_recollection} We recollect some facts about the metaplectic extension of $G$, for details we refer the reader to \cite[§2]{witthaus_meta}.
The metaplectic cover of $G$ is the topological central extension
\[
1\to \mu_2=\{\pm 1\} \to \wt{G}\to G\to 1
\]
defined by the $2$-cocycle $\sigma\colon G\times G\to \mu_2$,
\begin{equation}\label{cocycle}	\sigma(g_1,g_2)=\left(\frac{\mathfrak{c}(g_1g_2)}{\mathfrak{c}(g_1)},\frac{\mathfrak{c}(g_1g_2)}{\mathfrak{c}(g_2)}\det(g_1)\right) \text{ for all $(g_1,g_2)\in G\times G$},
\end{equation}
where $\mathfrak{c}\left(\begin{mat}
	a & b\\ c & d
\end{mat}\right)=\begin{cases}
	c \text{ if $c\neq 0$}\\
	d \text{ if $c=0$,}
\end{cases}$ and $(-,-)\colon \bQ_p^{\times}\times \bQ_p^{\times}\twoheadrightarrow \bQ_p^{\times}/S\times \bQ_p^{\times}/S\to \mu_2$ is the quadratic Hilbert symbol, explicitly given by $(a,b)=\omega\left((-1)^{v(a)v(b)} \frac{b^{v(a)}}{a^{v(b)}}\right)^{\frac{p-1}{2}}$ for all $a,b\in \bQ_p^{\times}$, where $\omega$ is the mod-$p$ cyclotomic character.

\begin{notation}
	(a) As a set, $\wt{G}=G\times \mu_2$ and we write elements as pairs $(g,\zeta)$ for $g\in G$ and $\zeta\in \mu_2$.
	
	(b) The preimage of a subset $C\subset G$ in $\wt{G}$ will be denoted by $\wt{C}$.
\end{notation}

The map
\begin{align*}
	K\times \mu_2&\xrightarrow{\cong} \wt{K}\\
	(g,\zeta) &\mapsto \begin{cases}
		(g,\zeta(c,d\det(g)^{-1})) \text{ if $c\in p\bZ_p\setminus \{0\}$}\\
		(g,\zeta) \text{ otherwise}
	\end{cases}
\end{align*}
defines a group isomorphism, which we fix throughout and use to view subgroups of $K$ as subgroups of $\wt{K}$.

As a group, the center $Z(\wt{G})=Z^{\sq}\times \mu_2\cong S\times \mu_2$ is equal to the product of the subgroup of squares in the center $Z$ of $G$ and $\mu_2$.\\

Let $\iota\colon \mu_2\hookrightarrow k^{\times}$ be the non-trivial character. Given a subgroup $C\subset G$, one has the product decomposition
\[
\Mod^{\sm}_{\wt{C}}(k)=\Mod^{\sm}_{\wt{C},\iota}(k)\times \Mod^{\sm}_C(k),
\]
where the first factor is the full subcategory consisting of the \textit{genuine} representations, i.e.\ those on which $\mu_2$ acts via $\iota$, and the second factor is viewed as a full subcategory via inflation.

\begin{notation}
	Given a subgroup $C\subset G$ together with a group isomorphism $\wt{C}\cong C\times \mu_2$, we denote the induced equivalence of categories $\Mod^{\sm}_C(k)\simeq \Mod^{\sm}_{\wt{C},\iota}(k)$ by $\pi\mapsto \pi\boxtimes \iota$.
\end{notation}

We will be interested in $\Mod^{\sm}_{\wt{G},\iota}(k)$, the absolutely irreducible objects of which have been described in \cite{witthaus_meta}: they are either genuine principal series or genuine supersingular representations. By definition, a genuine principal series on a $k$-vector space is of the form $\Ind^{\wt{G}}_{\wt{B}}(\tau)$ for some smooth irreducible genuine $k$-representation $\tau$ of $\wt{T}$, while a genuine supersingular representation is a smooth irreducible $k$-representation of $\wt{G}$ not occurring as a subquotient of, or equivalently being isomorphic to, a principal series.

\begin{prop}\label{PS}
	Let $B_S=\begin{mat}
		\bQ_p^{\times} & \bQ_p\\ 0 & 1
	\end{mat}\supset \begin{mat}
		\bQ_p^{\times} & 0\\ 0 & S
	\end{mat}=T_S$.
	\begin{enumerate}
		\item[{\rm (i)}] As groups, $\wt{B}_S=B_S\times \mu_2\supset T_S\times \mu_2=\wt{T}_S$;
		\item[{\rm (ii)}] The following three sets, considered up to isomorphism, are in bijection:
		\begin{enumerate}
			\item[{\rm (a)}] $\left\{\text{Smooth characters } T^{\sq}\to k^{\times}
			\right\};$
			\item[{\rm (b)}] $\left\{
			\text{Smooth abs.\ irred.\ genuine }
			\text{representations of $\wt{T}$ on $k$-vector spaces}\right\};$
			\item[{\rm (c)}] $\left\{
			\text{Abs.\ irred.\ genuine }
			\text{principal series of $\wt{G}$ on $k$-vector spaces}\right\}.$
		\end{enumerate}
		The bijection from (a) to (b) is given by $\chi\mapsto \Ind^{\wt{T}}_{\wt{T}_S}(\chi'\boxtimes \iota)$ for any extension $\chi'$ of $\chi$ to $T_S$, while the bijection between (b) and (c) is induced by the functor $\Ind^{\wt{G}}_{\wt{B}}(-)$ composed with inflation along $\wt{B}\twoheadrightarrow \wt{T}$.
	\end{enumerate}
\end{prop}

\begin{proof}
	Part (i) is Lemma 2.4 (ii), while part (ii) is the content of Remark 4.17, Definition-Proposition 4.16 and Proposition 4.20 in \cite{witthaus_meta}.
\end{proof}

The absolutely irreducible genuine principal series representations may also be described in terms of the universal module of the spherical Hecke algebras, which is also suitable to parametrize the supersingular objects: 

For $0\leq r\leq p-1$, consider the irreducible $\GL_2(\bF_p)$-representation $\Sym^r(k^2)$, where $k^2$ denotes the standard representation (base changed to $k$), viewed as a smooth irreducible representation of $K$ via inflation. We extend the action to $KZ^{\sq}$ by letting the central matrix $p^2$ act trivially. Via the identification $\wt{K}Z(\wt{G})\cong KZ^{\sq}\times \mu_2$, we may thus form the smooth irreducible genuine $k$-representation $\Sym^r(k^2)\boxtimes \iota$ of $\wt{K}Z(\wt{G})$. According to \cite[Proposition 5.2]{witthaus_meta}, the spherical Hecke algebra 
\[
\End_{\wt{G}}(\cInd^{\wt{G}}_{\wt{K}Z(\wt{G})}(\Sym^r(k^2)\boxtimes \iota))=k[\wt{T}]
\]
is a polynomial ring in a single Hecke operator $\wt{T}$.

\begin{prop}[{\cite[Theorem 5.13]{witthaus_meta}}]\label{cokernels}
	Let $0\leq r\leq p-1$ and $\lambda\in k$.
	
	\begin{enumerate}
		\item[{\rm (i)}] If $\lambda \neq 0$, then 
		\[
		\cInd^{\wt{G}}_{\wt{K}Z(\wt{G})}(\Sym^r(k^2)\boxtimes \iota)/(\wt{T}-\lambda)\cong \Ind^{\wt{G}}_{\wt{B}_S}(\psi'\boxtimes \iota),
		\]
		where $\psi'\colon T_S\to k^{\times}$ is any extension of the character $\psi\colon T^{\sq}\to k^{\times}$ defined by $\psi|_{K\cap T^{\sq}}=(\mathbf{1}\otimes \omega^r)|_{K\cap T^{\sq}}$, $\psi\left(\begin{mat}
			p^{-2} & 0\\0 & 1
		\end{mat}\right)=\lambda$ and $\psi\left(\begin{mat}
			p^{-2} &0\\0 & p^{-2}
		\end{mat}\right)=1$.
		\item[{\rm (ii)}] If $\lambda = 0$, then $\cInd^{\wt{G}}_{\wt{K}Z(\wt{G})}(\Sym^r(k^2)\boxtimes \iota)/(\wt{T})$ is an extension of two non-isomorphic absolutely irreducible genuine supersingular representations. The extension splits if and only if $r=\frac{p-1}{2}$.
	\end{enumerate}
\end{prop}

\begin{definition}
	For $0\leq r\leq p-1$, $r\neq \frac{p-1}{2}$, $\lambda\in k$ and a smooth character $\eta\colon \bQ_p^{\times}\to k^{\times}$, we define $\tilde{\pi}(r,\lambda,\eta)$ to be the unique irreducible quotient of $\cInd^{\wt{G}}_{\wt{K}Z(\wt{G})}(\Sym^r(k^2)\boxtimes \iota)/(\wt{T}-\lambda)\otimes \eta\circ {\det}$. If $\eta=\mathbf{1}$ is the trivial character, we simply write $\tilde{\pi}(r,\lambda)$.
\end{definition}

\begin{remark}
	The extension in Proposition \ref{cokernels} (ii) is an extension of $\tilde{\pi}(r,0)$ by $\tilde{\pi}(p-1-r,0,\omega^r)$.
\end{remark}

\begin{prop}\label{class_irred}
	Every smooth absolutely irreducible genuine $k$-representation of $\wt{G}$ is of the form $\tilde{\pi}(r,\lambda,\eta)$ for some $0\leq r\leq p-1$, $r\neq \frac{p-1}{2}$, $\lambda\in k$ and a smooth character $\eta\colon \bQ_p^{\times}\to k^{\times}$.
	\begin{enumerate}
		\item[{\rm (i)}] If $\lambda\neq 0$, then $\tilde{\pi}(r,\lambda,\eta)$ is a principal series. Moreover, $\tilde{\pi}(r,\lambda,\eta)\cong \tilde{\pi}(r',\lambda')$ if and only if $r\equiv r' \bmod \frac{p-1}{2}$, $\eta|_{\bZ_p^{\times,\sq}}=\mathbf{1}$, $\eta(p^4)=1$ and $\lambda'=\lambda\eta(p^{-2})$.
		\item[{\rm (ii)}] If $\lambda=0$, then $\tilde{\pi}(r,0,\eta)$ is supersingular. Moreover, $\tilde{\pi}(r,0,\eta)\cong \tilde{\pi}(r',\lambda')$ if and only if $\lambda'=0$, $\eta(p^4)=1$ and either $r'=r$ and $\eta|_{\bZ_p^{\times,\sq}}=\mathbf{1}$ or $r=\begin{cases}
			\frac{p-1}{2}-r'\text{ if $0<r'<\frac{p-1}{2}$}\\
			\frac{3(p-1)}{2}-r' \text{ if $\frac{p-1}{2}<r'<p-1$}
		\end{cases}$ and $\eta|_{\bZ_p^{\times,\sq}} = \omega^{r'}|_{\bZ_p^{\times,\sq}}$.
	\end{enumerate}
\end{prop}

\begin{proof}
	This follows from propositions \ref{cokernels}, \ref{PS} and \cite[Corollary 5.16]{witthaus_meta}.
\end{proof}

\subsection{Some finiteness result}\label{section_finiteness_results} Recall that a smooth representation is admissible if the subspace of $C$-fixed vectors is finite dimensional for every open subgroup $C$. The domain of the extension of Colmez's functor will be the category $\Mod^{\fadm}_{\wt{G}}(k)$ of finitely generated admissible $k$-representations of $\wt{G}$. We need to establish a certain finiteness result for this category in order to define the desired functor. The discussion for the full subcategory $\Mod^{\fadm}_G(k)$ by Emerton \cite{emerton_coherent_rings} will serve as a guideline.\\

By \cite[Proposition 2.2.10]{Emerton_OP1}, the category $\Mod^{\adm}_{\wt{G}}(k)$ of smooth admissible $k$-representations of $\wt{G}$ is abelian.

\begin{lem}
	An admissible object in $\Mod^{\sm}_{\wt{G}}(k)$ is finitely generated if and only if it is of finite length. In particular, the category $\Mod^{\fadm}_{\wt{G}}(k)$ is a Serre subcategory of $\Mod^{\adm}_{\wt{G}}(k)$ and
	hence is abelian.
\end{lem}

\begin{proof}
	We only need to prove that an object of $\Mod^{\fadm}_{\wt{G}}(k)$ is of finite length. By \cite[Theorem 5.17]{witthaus_meta}, this holds true if $k$ is replaced by an algebraic closure $\bar{k}$. Let now $\pi\in \Mod^{\fadm}_{\wt{G}}(k)$. We claim that every suboject $\tau\subset \pi$ is again finitely generated (and admissible). Indeed, $\pi\otimes_k \bar{k}$ is finitely generated and admissible, hence of finite length and so must be the subobject $\tau\otimes_k \bar{k}$, which is therefore finitely generated over the group ring $\bar{k}[\wt{G}]$. By faithful flatness, we deduce that $\tau$ is finitely generated over $k[\wt{G}]$ proving the claim. We may now choose a (possibly infinite) decomposition series of $\pi$, which must however be finite by the result over $\bar{k}$ and faithfulness of the functor $-\otimes_k \bar{k}$, i.e.\ $\pi$ is of finite length.
\end{proof}

Denote by $P=\begin{mat}
	\bQ_p^{\times} & \bQ_p\\ 0 & 1
\end{mat}$ the mirabolic subgroup, which we view as a subgroup of $\wt{G}$ via Proposition \ref{PS} (i). The submonoid $P^+=\begin{mat}
	\bZ_p\setminus \{0\} & \bZ_p\\ 0 &1
\end{mat}$ is generated by the element $F=\begin{mat}
	p & 0\\0 & 1
\end{mat}$ and the submonoids $\Gamma=\begin{mat}
	\bZ_p^{\times} & 0\\0 & 1
\end{mat}$, $\begin{mat}
	1 & \bZ_p\\ 0 & 1
\end{mat}$.

Identifying the completed group ring of $\begin{mat}
	1 & \bZ_p\\ 0 & 1
\end{mat}$ with the power series ring $k\llbracket X\rrbracket$ in the variable $X=\begin{mat}
	1 & 1\\0 & 1
\end{mat}-1$, any smooth $P^+$-representation on a $k$-vector space is naturally a module over the non-commutative ring $k\llbracket X\rrbracket[F,\Gamma]$ defined by the relations $Ff(X)=f(X^p)F$, $\gamma f(X)=f((1+X)^{\gamma}-1) \gamma$ and $F\gamma=\gamma F$ for all $f(X)\in k\llbracket X\rrbracket$ and $\gamma\in \Gamma\cong \bZ_p^{\times}$. 

\begin{lem}
	For $\pi\in \Mod^{\adm}_{\wt{G}}(k)$, the $k[F]$-module $\pi/X\pi$ is torsion, i.e.\ each element in it is killed by a polynomial in $F$ with $k$-coefficients.
\end{lem}

\begin{proof}
	Writing an admissible representation as the filtered colimit of its finitely generated subrepresentations, we may assume that $\pi\in \Mod^{\fadm}_{\wt{G}}(k)$. Any polynomial in $\bar{k}[F]$ is a factor of a polynomial in $k[F]$, hence we may replace $k$ by $\bar{k}$. By \cite[Theorem 5.17]{witthaus_meta}, an object in $\Mod^{\fadm}_{\wt{G}}(\bar{k})$ is of finite length, and we are reduced to proving the statement for smooth irreducible $\bar{k}$-representation $\pi$ of $\wt{G}$. If $\mu_2$ acts trivially, then this is a special case of \cite[Proposition 4.5]{emerton_coherent_rings}. In the genuine case, this is Corollary 4.56 (supersingular) and Corollary 4.32 (principal series) in \cite{witthaus_meta}.
\end{proof}

\begin{definition}\label{def_pi_M} Given a smooth $k$-representation $\pi$ of $P$ and a $\Gamma$-stable subspace $M\subset \pi$, define $\pi_M\subset \pi$ to be the $k\llbracket X\rrbracket[F]$-submodule of $\pi$ generated by $\pi_0$ (which is in fact a $k\llbracket X\rrbracket[F,\Gamma]$-submodule or equivalently a $P^+$-subrepresentation).
\end{definition}

Utilizing the previous lemma, one proves the following finiteness result just as Emerton \cite[Theorem 4.7]{emerton_coherent_rings} does for $\GL_2$ over an unramified extension of $\bQ_p$.

\begin{prop}\label{prop_finiteness}
	Let $\pi\in \Mod^{\adm}_{\wt{G}}(k)$ and let $M\subset \pi$ be a finite dimensional $\Gamma$-stable subspace. Then $\pi_M$ finitely presented and of finite length over $k\llbracket X\rrbracket[F]$ and the $k$-subspace of $X$-torsion elements $\pi_M[X]$ is finite dimensional.
\end{prop}

\begin{definition}\label{def_G(pi)}
	For $\pi\in \Mod^{\fadm}_{\wt{G}}(k)$, define $\mathcal{G}(\pi)$ to be the directed set of finite
	dimensional $\wt{K}\wt{Z}$-stable subspaces of $\pi$ generating the latter as a $\wt{G}$-representation, endowed
	with the partial order given by inclusion.
\end{definition}

\begin{remark}
	For $\pi\in \Mod^{\fadm}_{\wt{G}}(k)$, the set $\mathcal{G}(\pi)$ is non-empty: a finite number of generators will be fixed by a small congruence subgroup $K_n=1+p^nM_{2\times 2}(\bZ_p)$, so that $\pi$ is generated as a $\wt{G}$-representation by $\pi^{K_n}$, which is finite dimensional, by admissibility, and stable under the action of $\wt{K}\wt{Z}$.
\end{remark}

	\section{Galois representations and $(\varphi,\Gamma)$-modules}
	
	One of the key ingredients for relating the smooth mod-$p$ representation theory of $\wt{G}$ to Galois representations, as we will do in the next section, is Fontaine's equivalence giving a semi-linear algebraic description of Galois representations in terms of $(\varphi,\Gamma)$-modules. The purpose of this section is to remind the reader of this categorical equivalence as well as to recollect the known description of the $(\varphi,\Gamma)$-module corresponding to an absolutely irreducible Galois representation.

\subsection{Notations}	For a field $L$ with separable closure $L^{\operatorname{sep}}$, write $\mathscr{G}_L=\Gal(L^{\operatorname{sep}}/L)$ for its absolute Galois group and denote by $\Mod^{\fin}_{\mathscr{G}_L}(k)$ the category of smooth $\mathscr{G}_L$-representations on finite dimensional $k$-vector spaces.

We fix an algebraic closure $\ol{\bQ}_p$ of $\bQ_p$, and for $n\geq 1$ we let $\bQ_{p^n}$ denote the unique degree-$n$ unramified extension of $\bQ_p$ (inside $\ol{\bQ}_p$). We define
\[
\omega_n\colon \mathscr{G}_{\bQ_{p^n}}\twoheadrightarrow \bF_{p^n}^{\times}, \hspace{0.2cm} g(\sqrt[p^n-1]{-p})=\omega_n(g)\sqrt[p^n-1]{-p} \text{ \phantom{m}for all $g\in \mathscr{G}_{\bQ_{p^n}}$.}
\]
to be Serre's fundamental character of level $n$ corresponding to the uniformizer $-p$ of $\bQ_p$. In case $n=1$, $\omega_1=\omega$ is simply the mod-$p$ cyclotomic character. For $\lambda\in k^{\times}$, let $\mu_{\lambda}\colon \GQp\to k^{\times}$ be the unique unramified character mapping the geometric Frobenius to $\lambda$.

Finally, normalizing local class field theory so that $p$ maps to a geometric Frobenius, we will view smooth characters of $\GQp$ as smooth characters of $\bQ_p^{\times}$.

\subsection{Absolutely irreducible Galois representations} Following Berger \cite{berger_onsome}, we say that an integer $1\leq h\leq p^n-2$ is \textit{primitive} (relative to $n\geq 1$) if $h$ is not divisible by $(p^n-1)/(p^d-1)$ for any $d<n$ dividing $n$. In this case, Mackey's decomposition implies that the Galois representation $\Ind^{\mathscr{G}_{\bQ_p}}_{\mathscr{G}_{\bQ_{p^n}}}(\omega_n^h)$ is absolutely irreducible. Note that it is not only defined over $\bF_{p^n}$ but already over $\bF_p$. 

\begin{prop}[{\cite[Corollary 2.1.5]{berger_onsome}}]\label{irred_Galois}
	Every $n$-dimensional absolutely irreducible object in $\Mod^{\fin}_{\GQp}(k)$ is isomorphic to $\Ind^{\mathscr{G}_{\bQ_p}}_{\mathscr{G}_{\bQ_{p^n}}}(\omega_n^h)\otimes \mu_{\lambda}$ for some primitive $1\leq h\leq p^n-2$ and $\lambda\in \bar{k}^{\times}$ with $\lambda^n\in k^{\times}$.
\end{prop}

\begin{remark}
	For a field extension $\ell/k$ and $V_1,V_2\in \Mod^{\fin}_{\GQp}(k)$, one has $\Hom_{k[\GQp]}(V_1,V_2)\otimes_k \ell=\Hom_{\ell[\GQp]}(V_1\otimes_k \ell,V_2\otimes_k \ell)$. In particular, there is no harm in extending scalars to a finite extension of $k$ when dealing with absolutely irreducible Galois representations.
\end{remark}

\subsection{Fontaine's equivalence}\label{section_Fontaines_equiv} Let $\bQ_p(\mu_{p^{\infty}})$ be the subfield of $\ol{\bQ}_p$ obtained by adjoining all $p$-power roots of unity to $\bQ_p$. Consider the short exact sequence
\[
0\to \mathscr{G}_{\bQ_p(\mu_{p^{\infty}})}\to \mathscr{G}_{\bQ_p}\to \Gamma:=\Gal(\bQ_p(\mu_{p^{\infty}})/\bQ_p)\to 0.
\]
The cyclotomic character $\chi_{\operatorname{cyc}}\colon \GQp\twoheadrightarrow \bZ_p^{\times}$ induces an isomorphism of topological groups $\chi_{\operatorname{cyc}}\colon \Gamma=\Gal(\bQ_p(\mu_{p^{\infty}})/\bQ_p)\cong \bZ_p^{\times}$. For the kernel of the sequence, note that completing $\bQ_p(\mu_{p^{\infty}})$ $p$-adically results in a perfectoid field $\bQ_p(\mu_{p^{\infty}})^{\wedge}$, whose tilt (as defined in \cite[Lemma 3.4 (ii)]{scholze}) is computed to be
\begin{equation*}
	\left(\bQ_p(\mu_{p^{\infty}})^{\wedge}\right)^{\flat}=\left(\varprojlim_{x\mapsto x^p} \bZ_p[\mu_{p^{\infty}}]/p\right)[1/X] \cong \bF_p\llbracket X^{1/p^{\infty}}\rrbracket [1/X]=\left(\bF_p((X))^{\operatorname{perf}}\right)^{\wedge},
\end{equation*}
where $X=(\zeta_{p^n}- 1)_{n\geq 1}$, for a compatible system $\{\zeta_{p^n}\}_{n\geq 1}$ of $p$-power roots of unity. From now on we fix the choice of a compatible system of $p$-power roots and thus the above isomorphism. By the tilting equivalence \cite[Theorem 3.7]{scholze}, we have identifications
\begin{equation*}
	\mathscr{G}_{\bQ_p(\mu_{p^{\infty}})}\cong \mathscr{G}_{\bQ_p(\mu_{p^{\infty}})^{\wedge}}\cong \mathscr{G}_{\bF_p((X^{1/p^{\infty}}))}\cong \mathscr{G}_{\bF_p((X))^{\operatorname{perf}}}\cong \mathscr{G}_{\bF_p((X))},
\end{equation*}
where the first and second-to-last isomorphism are induced by taking completions (Krasner's Lemma).

Via functoriality of tilting and continuity, the Galois group $\mathscr{G}_{\bQ_p}$ acts on $(\ol{\bQ}_p^{\wedge})^{\flat}$ and it preserves the subfield $\bF_p((X))$: $g(X)=(X+1)^{\chi_{\operatorname{cyc}}(g)}-1$, for $g\in \GQp$.

\begin{remark}
	In particular, there is no (serious) clash of notation for $\Gamma$ and $X$ appearing in this section and Section \ref{section_finiteness_results}: in both cases $\Gamma$ is identified with $\bZ_p^{\times}$ and the two actions on $X$ agree. In fact, this compatibility is what allowed Colmez to define his functor as we will see below.
\end{remark}

Since the completed algebraic closure is algebraically closed, we obtain that $\bF_p((X))^{\operatorname{sep}}\subset (\ol{\bQ}_p^{\wedge})^{\flat}$, which is then also $\mathscr{G}_{\bQ_p}$-stable. We will let $\varphi$ denote the Frobenius endomorphism $a\mapsto a^p$ on $\bF_p((X))^{\operatorname{sep}}$ and $\bF_p((X))$. In particular, we have the commuting actions of $\varphi$ and $\Gamma$ on $\bF_p((X))$. 

Applying the functor $k\otimes_{\bF_p} -$, we obtain similar actions of $(\varphi,\Gamma)$ on $k((X))$.

\begin{definition}[{\cite[3.3]{Fontaine}}]\label{def_phi_gamma} An \textit{\'{e}tale $(\varphi,\Gamma)$-module over $k((X))$} is a finite dimensional $k((X))$-module $D$ endowed with a $\varphi$-semilinear endomorphism $\varphi_D$ and a continuous\footnote{After fixing a basis of $D$, we can identify $D\cong k((X))^n$ for some $n\geq 0$ and we take the topology for which the $k\llbracket X\rrbracket$-submodule $k\llbracket X\rrbracket^n$ endowed with the $X$-adic topology is an open subgroup.} semilinear $\Gamma$-action commuting with $\varphi_D$, such that the linearized map $\varphi_D\colon \varphi^{\ast}D\to D$ is an isomorphism.
\end{definition}

We denote the resulting abelian category by $\phigamma$. It admits tensor products $-\otimes_{k((X))}-$ and duals $(-)'$: while the former is formed factorwise, the latter is uniquely characterized by $D'=\Hom_{k((X))}(D,k((X)))$ on underlying $k((X))$-vector spaces and requiring the evaluation map $D'\otimes_{k((X))} D\to  k((X))$ to be $(\varphi,\Gamma)$-equivariant.

\begin{def-prop}[{\cite[3.4.3]{Fontaine}}]\label{fontaine_equiv} The functors
	\begin{align*}
		\Phi\Gamma^{\operatorname{\acute{e}t}}_{k((X))} &\rightleftarrows \Mod^{\fin}_{\mathscr{G}_{\bQ_p}}(k)\\
		D&\mapsto \mathscr{V}(D)=(D\otimes_{\bF_p((X))} \bF_p((X))^{\operatorname{sep}})^{\varphi_D\otimes \varphi=1}\\
		\mathscr{D}(V)=(V\otimes_{\bF_p} \bF_p((X))^{\operatorname{sep}})^{\mathscr{G}_{\bQ_p(\mu_{p^{\infty}})}} &\mapsfrom V
	\end{align*}
	define an equivalence of categories respecting tensor products and duals. Here, the $\mathscr{G}_{\bQ_p}$-action on $\mathscr{V}(D)$ is induced by the diagonal action\footnote{$\mathscr{G}_{\bQ_p}$ acts on $D$ through the quotient $\Gamma$ and on $\bF_p((X))^{\operatorname{sep}}$ as explained prior to Definition \ref{def_phi_gamma}.}. In the definition of $\mathscr{D}(V)$, the $\mathscr{G}_{\bQ_p(\mu_{p^{\infty}})}\cong \mathscr{G}_{\bF_p((X))}$-invariants are taken with respect to the diagonal action, so one obtains a natural action of the quotient $\Gamma$; the $\varphi$-action on $\mathscr{D}(V)$ is induced by acting on $\bF_p((X))^{\operatorname{sep}}$. The action of $k$ on either side is induced by the action of it on $D$ and $V$, respectively.
\end{def-prop}

\begin{example}[{\cite[Corollaire 3.7]{Breuil_diagram}}]\label{example_phi_gamma}
	Fix integers $n\geq 1$, $h\geq 0$ and let $\chi\colon \bQ_p^{\times}\to \bar{k}^{\times}$ be a character with $\chi(p)^n\in k^{\times}$, so that the Galois representation $V=\Ind^{G_{\bQ_p}}_{G_{\bQ_{p^n}}}(\omega_n^h)\otimes \chi$ is defined over $k$. The $(\varphi,\Gamma)$-module $\mathscr{D}(V)$ attached to $V$ is the $n$-dimensional vector space over $k((X))$ with basis $e_1,\ldots,e_n$ and
	\begin{align*}
		&\varphi(e_i)=e_{i+1} \text{ for $1\leq i \leq n-1$ and } \varphi(e_n)= \chi(p)^n X^{-h(p-1)}e_1\\
		&\gamma(e_i)=\chi(\gamma) \left(\frac{\omega(\gamma)X}{\gamma(X)}\right)^{h\frac{p^{i-1}(p-1)}{p^{n}-1}} e_i \text{ for all $1\leq i\leq n$.}
	\end{align*}
	In case $1\leq h\leq p^n-2$ is primitive relative to $n$ (so that $V$ is absolutely irreducible), this was also computed by Berger \cite[Theorem 2.1.6]{berger_onsome}.
\end{example}

\subsection{$(\psi,\Gamma)$-modules}\label{section_psi-gamma} We finish this section with a brief discussion of $(\psi,\Gamma)$-modules following \cite[II.]{colmez1}.

The field $k((X))$ has basis $\{(X+1)^{i}\}_{i=0,\ldots,p-1}$ over its subfield $k((X^p))=\varphi(k((X)))$. Thus, if $D$ is an \'{e}tale $(\varphi,\Gamma)$-module over $k((X))$, then every element $v\in D$ is uniquely of the form $v=\sum_{i=0}^{p-1} (X+1)^{i}\varphi(v_i)$ for some elements $v_i\in D$.

\begin{definition}[{\cite[II.3]{colmez1}}]
	Let $D\in \phigamma$. For $v=\sum_{i=0}^{p-1}(X+1)^{i}\varphi(v_i)\in D$ with $v_i\in D$ for $i=0,\ldots,p-1$, we define $\psi(v)=v_0$, resulting in an additive map $\psi\colon D\to D$ commuting with the $\Gamma$-action and satisfying $\psi\circ \varphi=\operatorname{id}_D$, $\psi(f(X)\varphi(v))=\psi(f(X)) v$ and $\psi(\varphi(f(X))v)=f(X)\psi(v)$ for all $f(X)\in k((X))$ and $v\in D$.
\end{definition}

The operator $\psi$ is functorial in the sense that a morphism $D_1\to D_2$ in $\phigamma$ is $\psi$-equivariant.

\begin{def-lem}[{\cite[Proposition II.4.2, Corollaire II.5.12]{colmez1}}] Let $D\in \phigamma$.
	\begin{enumerate}
		\item[{\rm (i)}] There exists a unique lattice\footnote{This means a finitely generated $k\llbracket X\rrbracket$-stable subspace containing a basis of $D$ over $k((X))$.} $D^{\sharp}\subset D$, which is stable under the action of $\psi$ such that $\psi$ is surjective on it and for each $v\in D$, there is some $n\geq 1$ so that $\psi^n(v)\in D^{\sharp}$.
		
		\noindent Moreover, if $N\subset D$ is any $\psi$-stable lattice on which $\psi$ is surjective, then $N\subset D^{\sharp}\subset X^{-1}N$.
		\item[{\rm (ii)}] There exists a unique minimal $\psi$-stable lattice $D^{\natural}$ of $D$. Moreover, $\psi$ is surjective on it.
	\end{enumerate}
\end{def-lem}

\begin{remark}[{\cite[Proposition II.4.6, Proposition II.5.17]{colmez1}}]\label{sharp_functorial}
	Note that the definitions of $\psi$, $D^{\sharp}$ and $D^{\natural}$ do not use the $\Gamma$-action, i.e.\ they are already defined for \'{e}tale $\varphi$-modules and on this category the assignments $D\mapsto D^{\sharp}$ and $D\mapsto D^{\natural}$ are functorial. In particular, if $D$ is a $(\varphi,\Gamma)$-module over $k((X))$, then the subspaces $D^{\natural}$ and $D^{\sharp}$ are $\Gamma$-stable (since the $\varphi$-action commutes with $\Gamma$). The resulting objects are instances of $(\psi,\Gamma)$-modules over $k\llbracket X\rrbracket$, the definition of which the reader can guess but we will not need.
\end{remark}

\begin{example}[{\cite[Lemme 1.1.2]{berger_french}}]\label{D_nat=D_sharp}
	\begin{enumerate}
	\item[{\rm (i)}] Let $D=k((X))e$ be the $(\varphi,\Gamma)$-module attached to a smooth character $\chi\colon \bQ_p^{\times}\to k^{\times}$ in the sense that $\varphi(e)=\chi(p)e$ and $\gamma(e)=\chi(\gamma)e$ for all $\gamma\in \Gamma\cong \bZ_p^{\times}$, i.e.\ $D=\mathscr{D}(\chi)$ via Fontaine's equivalence \ref{fontaine_equiv} (and local class field theory). Then $D^{\natural}=k\llbracket X\rrbracket e\subset X^{-1}k\llbracket X\rrbracket e = D^{\sharp}$.
	
	\item[{\rm (ii)}] Let $n\geq 2$ be an integer. If $V$ is an $n$-dimensional absolutely irreducible Galois representation of $\mathscr{G}_{\bQ_p}$ on a $k$-vector space, then $\mathscr{D}(V)^{\natural}=\mathscr{D}(V)^{\sharp} \subset \mathscr{D}(V)$. In fact, any non-trivial $\psi$-stable $k\llbracket X\rrbracket$-submodule of $\mathscr{D}(V)^{\sharp}$ must already be equal to $\mathscr{D}(V)^{\sharp}$.
	\end{enumerate}
\end{example}

\section{An extension of Colmez's functor} 
In \cite{colmez2}, Colmez defines a functor from the category of finitely generated smooth admissible $k$-representations of $G$ to the category of Galois representations of $\GQp$ on $k$-vector spaces. This functor has the property that the image of an object only depends on the restriction of that object to the mirabolic subgroup $P$. More formally, using Breuil's generalization \cite{Breuil_phi}, one may view Colmez's functor as defined on the category of smooth $P$-representations at the cost of replacing the target category of Galois representations by its ind-category. When composed with the restriction functor $\Mod^{\fadm}_{G}(k)\to \Mod^{\sm}_{P}(k)$, this more general functor takes values in the subcategory of compact objects, i.e.\ usual Galois representations, and recovers Colmez's functor. By Proposition \ref{PS} (i), we have an equality of groups $\wt{P}=P\times \mu_2$, so we may also precompose with the restriction functor $\Mod^{\fadm}_{\wt{G}}(k)\to \Mod^{\sm}_P(k)$ and the finiteness results of Section \ref{section_finiteness_results} will imply that the composite again takes values in the subcategory of compact objects, thus extending Colmez's functor to the whole category $\Mod^{\fadm}_{\wt{G}}(k)$.

\subsection{From smooth $P$-representations to Galois representations} We recollect the definition and basic properties of Breuil's functor, as found in \cite[§2]{Breuil_phi}, from the category of smooth $k$-representations of $P$ to the ind-category of Galois representations of $\mathscr{G}_{\bQ_p}$ on $k$-vector spaces or equivalently of $(\varphi,\Gamma)$-modules over $k((X))$. 

\begin{definition}[{\cite[p.\ 8]{Breuil_phi}}]\label{def_M(pi)} Let $\pi$ be a smooth $k$-representation of $P$. We let $\mathcal{M}(\pi)$ denote the set of finitely generated $k\llbracket X\rrbracket[F,\Gamma]$-submodules $\tau$ of $\pi$ such that the space $\tau[X]$ of $X$-torsion elements is finite dimensional. We turn $\mathcal{M}(\pi)$ into a filtered set by giving it the partial order induced by inclusion (cf.\ \cite[Lemme 2.1]{Breuil_phi}).
\end{definition}

Although not necessary for constructing the Galois object associated to a smooth representation of $P$, the following observation seems to fit here just fine: As in \cite[Definition 2.3.1]{Emerton_OP1}, we say that $\pi\in \Mod^{\sm}_{\wt{B}}(k)$ is \textit{locally $Z(\wt{G})$-finite} if for each $v\in \pi$, the action map $k[Z(\wt{G})]\to \End_k(k[Z(\wt{G})]v)$ of $Z(\wt{G})$ on the $Z(\wt{G})$-subrepresentation of $\pi$ generated by $v$ has finite dimensional image. Denote the full abelian subcategory of $\Mod^{\sm}_{\wt{B}}(k)$ consisting of locally $Z(\wt{G})$-finite representations by $\Mod^{\sm}_{\wt{B}}(k)_{Z(\wt{G})-\mathrm{fin}}$.

\begin{lem}\label{cofinal_Z_stable}
	For $\pi\in \Mod^{\sm}_{\wt{B}}(k)_{Z(\wt{G})-\mathrm{fin}}$, the subset of $\mathcal{M}(\pi)$ consisting of those elements which are stable under the action of $\wt{Z}$ is cofinal.
\end{lem}

\begin{proof}
	An element $\tau\in \mathcal{M}(\pi)$ is finitely generated over $k\llbracket X\rrbracket[F]$. Since $Z(\wt{G})\subset \wt{Z}$ has finite index and $\pi$ is locally $Z(\wt{G})$-finite, there exist $z_1,\ldots,z_n\in \wt{Z}$ such that $\sum_{z\in \wt{Z}} z\tau=\sum_{i=1}^n z_i\tau$. Each of the summands lies in $\mathcal{M}(\pi)$ and, by \cite[Lemme 2.1]{Breuil_phi}, so does the finite sum which contains $\tau$, so we are done.
\end{proof}

As in Section \ref{section_Fontaines_equiv}, we let $\varphi=\operatorname{id}_k\otimes (a\mapsto a^p)$ be the relative Frobenius endomorphism on $k\otimes_{\bF_p}\bF_p((X))=k((X))$, and we consider the $\Gamma$-action on it defined by $\gamma f(X)=f((X+1)^{\chi_{\operatorname{cyc}}(\gamma)}-1)$ for all $f(X)\in k((X))$ and $\gamma\in \Gamma$. Note that the subring $k\llbracket X\rrbracket\subset k((X))$ is stable under the action of $\varphi$ and $\Gamma$.

Let now $\pi$ be a smooth $k$-representation of $P$ and consider an element $\tau\in \mathcal{M}(\pi)$. Linearizing the action of $F$ on $\tau$, we obtain the $k\llbracket X\rrbracket$-linear map
\begin{equation}\label{linearized}
	F\colon \varphi^{\ast}\tau\to \tau.
\end{equation}
We wish to take the Pontryagin dual $(-)^{\vee}=\Hom_{k}(-,k)$ of it: We turn $\tau^{\vee}$ into a profinite $k\llbracket X\rrbracket[\Gamma]$-module by letting $k\llbracket X\rrbracket$ and $\Gamma$ act via $(g(X).f)(-)=f(g(X)(-))$ and $(\gamma.f)(-)=f(\gamma^{-1}(-))$ for all $f\in \tau^{\vee}$, $g(X)\in k\llbracket X\rrbracket$ and $\gamma\in \Gamma$. By smoothness, we can write $\tau=\varinjlim_d \tau[X^{p^d}]$. Now each $\tau[X^{p^d}]$ is a finite dimensional $k\llbracket X\rrbracket[\Gamma]$-submodule of $\tau$, so that $\tau^{\vee}=\varprojlim_d (\tau[X^{p^d}]^{\vee})$ is indeed a profinite $k\llbracket X\rrbracket[\Gamma]$-module.\\

The following isomorphism is a corrected version of the one defined by Breuil, see also \cite[Equation (14)]{herzig_co} and the corresponding footnote there.

\begin{lem}
	The map
	\begin{align*}
		(\varphi^{\ast}\tau)^{\vee}=\left(k\llbracket X\rrbracket\otimes_{\varphi,k\llbracket X\rrbracket} \tau \right)^{\vee}&\xrightarrow{\cong} k\llbracket X\rrbracket\otimes_{\varphi,k\llbracket X\rrbracket} \tau^{\vee}=\varphi^{\ast}(\tau^{\vee})\\
		f & \mapsto \sum_{i=0}^{p-1} (1+X)^{i} \otimes f\left((1+X)^{-i}\otimes (-)\right)
	\end{align*}
	is an $k\llbracket X\rrbracket$-linear isomorphism.
\end{lem}

\begin{proof}
	Since $k\llbracket X\rrbracket$ has basis $\{(1+X)^{j}\}_{0\leq j\leq p-1}$ over $k\llbracket X^p\rrbracket$, it suffices to show that the map is $k\llbracket X^p\rrbracket$-linear and equivariant for the basis elements, which is straight forward. Similarly, $\{(1+X)^{-j}\}_{0\leq j\leq p-1}$ is another basis, which implies that the map is injective and gives a recipe for constructing the inverse.
\end{proof}

Since $\tau$ is finitely generated over $k\llbracket X\rrbracket[F]$ and torsion over $k\llbracket X\rrbracket$ (by smoothness), the cokernel of (\ref{linearized}) is finitely generated and torsion over $k\llbracket X\rrbracket$, i.e.\ it is finite dimensional over $k$. Thus, taking Pontryagin duals and inverting $X$, we obtain a $k((X))$-linear injection
\begin{equation}\label{F_dual}
	F^{\vee}[1/X]\colon \tau^{\vee}[1/X]\hookrightarrow (\varphi^{\ast}\tau)^{\vee}[1/X]\cong \varphi^{\ast}(\tau^{\vee}[1/X]).
\end{equation}
Since $\tau[X]$ is finite dimensional, we dually have that $\tau^{\vee}/X\tau^{\vee}$ is finite dimensional, i.e.\ the profinite $k\llbracket X\rrbracket$-module $\tau^{\vee}$ is finitely generated by the topological Nakayama Lemma. Thus, $\tau^{\vee}[1/X]$ is finite dimensional over $k((X))$. Hence, the above injection $F^{\vee}[1/X]$ between finite dimensional $k((X))$-vector spaces of the same dimension must be an isomorphism. We now define $\varphi_{\tau^{\vee}[1/X]}$ to be the composition
\[
\varphi_{\tau^{\vee}[1/X]}\colon \tau^{\vee}[1/X]\xrightarrow{1\otimes \operatorname{id}} \varphi^{\ast}(\tau^{\vee}[1/X])\xrightarrow{\left(F^{\vee}[1/X]\right)^{-1}} \tau^{\vee}[1/X].
\]
Together with the action of $\Gamma$ on $\tau^{\vee}[1/X]$ induced by the contragredient action on $\tau^{\vee}$ defined above, we then obtain the $(\varphi,\Gamma)$-module $\tau^{\vee}[1/X]$ over $k((X))$, which is \'{e}tale by construction.

\begin{remark}\label{formula}
	The following formula appearing in \cite[Equation (156)]{herzig_co} will come in handy:
	\[
	\sum_{i=0}^{p-1} (1+X)^{i} \varphi_{\tau^{\vee}[1/X]}(((1+X)^{-i}f)\circ F)=f \text{ for all $f\in \tau^{\vee}[1/X]$.}
	\]
	It is enough to prove this equality after applying the injection (\ref{F_dual}), where it then follows from the definitions.
\end{remark}

By taking duals $(-)'$ in the category of \'{e}tale $(\varphi,\Gamma)$-modules, we obtain the assignment $\tau\mapsto (\tau^{\vee}[1/X])'$. After possibly applying Fontaine's equivalence \ref{fontaine_equiv}, which is compatible with duals, we may pass to the filtered colimit to arrive at the definition:

\begin{definition}[{\cite[p.\ 9]{Breuil_phi}}]\label{D_dual} Given a smooth $k$-representation $\pi$ of $P$, we put\footnote{Breuil defines a pro-object. Since we dualized, we obtain an ind-object.}
	\[
	\mathbf{D}(\pi):= \underset{\tau\in \mathcal{M}(\pi)}{``{\varinjlim}"} (\tau^{\vee}[1/X])' \in \Ind(\phigamma)
	\]
	and
	\[
	\mathbf{V}(\pi):=\underset{\tau\in \mathcal{M}(\pi)}{``{\varinjlim}"} \mathscr{V}(\tau^{\vee}[1/X])^{\vee} \in \Ind(\Mod^{\fin}_{\GQp}(k)),
	\]
	where the filtered colimit is taken in the respective (abelian) ind-category. 
\end{definition}

Note that we have fully faithful embeddings $\phigamma\hookrightarrow \Ind(\phigamma)$ and $\Mod^{\fin}_{\GQp}(k)\hookrightarrow \Ind(\Mod^{\fin}_{\GQp}(k))$ which define an equivalence onto the respective subcategory of compact objects.

\begin{prop}[{\cite[Proposition 2.7]{Breuil_phi}}]\label{exactness_properties}
	\begin{enumerate}
		\item[{\rm (i)}] The assignment $\pi\mapsto \mathbf{D}(\pi)$ is functorial, resulting in a covariant functor $\mathbf{D}\colon \Mod^{\operatorname{sm}}_{P}(k)\to \Ind(\phigamma)$.
		\item[{\rm (ii)}] The functor $\mathbf{D}$ is left-exact.
		\item[{\rm (iii)}] If $\pi\in \Mod^{\sm}_P(k)$ is $k[F]$-torsion, then $\mathbf{D}(\pi)=0$.
		\item[{\rm (iv)}] If $0\to \pi_1\to \pi_2\to \pi_3\to 0$ is an exact sequence in $\Mod^{\sm}_P(k)$ such that every finitely generated $k\llbracket X\rrbracket[F]$-submodule of $\pi_1$ has finite dimensional $X$-torsion, then the sequence $0\to \mathbf{D}(\pi_1)\to \mathbf{D}(\pi_2)\to \mathbf{D}(\pi_3)\to 0$ is exact.
	\end{enumerate}
\end{prop}

\begin{definition}\label{phi_gamma_twist}
	For a smooth character $\chi\colon \bQ_p^{\times}\to k^{\times}$ and a $(\varphi,\Gamma)$-module $D\in \phigamma$, define $D\otimes \chi\in \phigamma$ to be the $(\varphi,\Gamma)$-module defined by
	\begin{itemize}
		\item $D\otimes \chi=D$ on underlying $k((X))$-vector spaces
		\item $\varphi_{D\otimes \chi}=\chi(p)\varphi_D$
		\item $\gamma_{D\otimes \chi}=\chi(\gamma)\gamma_D$ for all $\gamma\in \Gamma$,
	\end{itemize}
	i.e.\ we twist the $(\varphi,\Gamma)$-action on $D$ by the character $\chi$.
\end{definition}

\begin{remark}\label{remark_twist_Br}
	\begin{enumerate}
		\item[{\rm (i)}] The twisting construction $D\mapsto D\otimes \chi$ defines an autoequivalence on $\phigamma$ and hence induces one on the ind-category $\Ind(\phigamma)$.
		\item[{\rm (ii)}] The functor $\mathbf{D}\colon \Mod^{\sm}_{P}(k)\to \Ind(\phigamma)$ is compatible with twists by characters: For $\pi\in \Mod^{\sm}_P(k)$, we have $\mathbf{D}(\pi\otimes \chi\circ \det)=\mathbf{D}(\pi)\otimes \chi$.
	\end{enumerate}
\end{remark}

\subsection{Calculating Galois representations}\label{section_Galois_calculate} For specific $\pi\in \Mod^{\sm}_P(k)$, or rather a specific $\tau\in \mathcal{M}(\pi)$, the Galois representation $\mathscr{V}(\tau^{\vee}[1/X])$ can be computed explicitly, which will be relevant for us when computing the image of a genuine supersingular representation under the extension of Colmez's functor. The results in this subsection are an elaborated version of parts of Florian Herzig's lecture notes \cite{herzig_notes}.

We will work in the following setup also appearing in \cite[§4.4.1]{witthaus_meta}.\\

\textbf{Setup.} Let $\pi\in \Mod^{\sm}_P(k)$. Suppose that we are given $k$-linearly independent $\Gamma$-eigenvectors $v_1,\ldots,v_n\in \pi[X]$ killed by $X$ with $\Gamma$-eigencharacters $\chi_1,\ldots,\chi_n$, satisfying
\begin{equation}\label{cycle_equation}
	X^{s_i}F(v_i)=c_i v_{i+1} \text{ for some $s_i\in \bZ_{\geq 0}$ and $c_i\in k^{\times}$}
\end{equation}
for all $1\leq i\leq n$.
Here and in what follows we understand the subscript $i$ to be the unique element in $\{1,\ldots,n\}$ congruent to $i$ modulo $n$. We define the $\Gamma$-stable subspace
\[
M=\bigoplus_{i=1}^n kv_i
\]
of $\pi$ and let $\pi_M=k\llbracket X\rrbracket[F]M \subset \pi$ be the $k\llbracket X\rrbracket[F]$-submodule generated by $M$ as in Definition \ref{def_pi_M}. Finally, for $1\leq i\leq n$, we put $\pi_i=k\llbracket X\rrbracket[F^n]v_i\subset \pi_M.$ The subspaces $\pi_M$ and $\pi_i$, for $1\leq i\leq n$, are stable under the action of $\Gamma$. We would like to understand the $(\varphi,\Gamma)$-action on $\pi_M^{\vee}[1/X]$. If $s_i=0$ for all $1\leq i\leq n$, then $\pi_M=M$ is finite dimensional over $k$ and hence its localized Pontryagin dual vanishes, which is why we from now on assume that $s_j>0$ for some $1\leq j\leq n$.

\begin{lem}[{\cite[Lemma 4.41]{witthaus_meta}}]\label{basic_lemma}
	\begin{enumerate}
		\item[{\rm (i)}] As $k\llbracket X\rrbracket[F^n,\Gamma]$-modules, $\pi_M=\bigoplus_{i=1}^n \pi_i$.
		\item[{\rm (ii)}] For each $1\leq i\leq n$, the $k\llbracket X\rrbracket$-module $\pi_i^{\vee}$ is free of rank one.
		\item[{\rm (iii)}] Assume that the following condition is satisfied: If for some $1\leq i\neq j\leq n$, the vectors $v_i,v_j$ have the same $\Gamma$-eigencharacter, then $s_{i}\neq s_j$. Then $\pi_M$ is an irreducible $P^+$-representation.
	\end{enumerate}
\end{lem}

By parts (i) and (ii) of the lemma, we may write $\pi_M^{\vee}=\oplus_{i=1}^n \pi_i^{\vee}$ as $k\llbracket X\rrbracket$-modules with each $\pi_i^{\vee}$ being free of rank one: namely, a basis element of it is given by any $k$-linear function $\pi_i\to k$ not vanishing at $v_i$ (by Nakayama's Lemma). In order to compute the action of $\varphi$ and $\Gamma$, we would like to exhibit a particular basis element, for which we write $\pi_i=\varinjlim_m \pi_{i,m}$ as the increasing union of the subspaces $\pi_{i,m}=k\llbracket X\rrbracket F^{nm}v_i$. Letting $c=\prod_{i=1}^n c_i$ be the product of the constants appearing in (\ref{cycle_equation}), $e(i)=\sum_{j=0}^{n-1} p^{n-1-j}s_{i+j}$ and defining $e(i)_m=e(i)\frac{1-p^{nm}}{1-p^n}$ for all $m\geq 1$, an iteration of the equation(s) just cited shows that the map
\begin{align}\label{iso_pi_{i,m}}
	k\llbracket X\rrbracket/(X^{e(i)_{m}+1}) & \xrightarrow{\cong} \pi_{i,m}\\
	f(X)& \mapsto c^{-m}f(X)F^{nm}v_i\nonumber
\end{align}
is an isomorphism, see also \cite[Lemma 4.43, Equation (22)]{witthaus_meta}. We now define $f_i\in \pi_i^{\vee}=\varprojlim_m \pi_{i,m}^{\vee}$ by requiring
\[
f_i|_{\pi_{i,m}}\colon \pi_{i,m}\cong k\llbracket X\rrbracket/(X^{e(i)_{m}+1})\twoheadrightarrow k,
\]
to be the composition of the inverse of the isomorphism (\ref{iso_pi_{i,m}}) with the projection  $\sum_{d=0}^{e_i(m)}a_d X^d\mapsto a_{e(i)_m}$, for all $m\geq 1$. We have $f_i(v_i)=1\neq 0$ and so $f_i$ indeed is a basis element of $\pi_i^{\vee}$. Thus, the set $\{f_1,\ldots,f_n\}$ is a basis of the free $k\llbracket X\rrbracket$-module $\pi_M^{\vee}$.

\begin{lem}\label{phi_Gamma_description}
	For all $1\leq i\leq n$,
	\[
	\varphi(f_i)\in c_i^{-1} (1+Xk\llbracket X\rrbracket)X^{s_i-(p-1)}f_{i+1}
	\]
	and
	\[
	\gamma(f_i)\in \chi_i(\gamma)^{-1}(1+Xk\llbracket X\rrbracket)f_i, \text{ for all $\gamma\in \Gamma$,}
	\]
	where $\chi_i$ is the $\Gamma$-eigencharacter of $v_i$.
\end{lem}

\begin{proof}
	We follow the proof of \cite[Proposition 3.2.4.2]{herzig_co}. Let $1\leq i\leq n$. We claim that for all $j\geq 0$, 
	\[
	F^{\vee}((X^{s_i+j}f_{i+1}))=(X^{s_i+j}f_{i+1})\circ F\in \begin{cases}
		c_i(1+Xk\llbracket X\rrbracket)f_i \text{ if $j=0$}\\
		Xk\llbracket X\rrbracket f_i \text{ else.}
	\end{cases}
	\]
	Indeed, the first equality is just the definition of $F^{\vee}$, and we have $F(\pi_i)\subset \pi_{i+1}$ (which follows by an iteration of (\ref{cycle_equation}), see \cite[Lemma 4.43 (i)]{witthaus_meta}). In particular, $(X^{s_i+j}f_{i+1})\circ F\in \pi_i^{\vee}$, i.e.\ there exists a unique $h(X)\in k\llbracket X\rrbracket$ such that $(X^{s_i+j}f_{i+1})\circ F=h(X)f_i$. Since $v_i$ and $v_{i+1}$ are killed by $X$, evaluation at $v_i$ gives
	\[
	h(0)=f_i(h(X)v_i)=(h(X)f_i)(v_i)=f_{i+1}(X^{s_i+j}F(v_i))=\begin{cases}
		c_{i} \text{ if $j=0$}\\
		0 \text{ if $j>0$,}
	\end{cases}
	\]
	where the last equality uses the relation (\ref{cycle_equation}). This proves the claim. In particular, for all $j\geq 0$, $F^{\vee}((1+X)^{-j}X^{s_i}f_{i+1})\in c_i(1+Xk\llbracket X\rrbracket)f_i$. Applying $\varphi$, we obtain
	\[
	\varphi(F^{\vee}((1+X)^{-j}X^{s_i}f_{i+1}))\in c_i(1+X^pk\llbracket X\rrbracket)\varphi(f_i),
	\]
	say it is equal to $c_i(1+X^ph_j(X))\varphi(f_i)$ for some $h_j(X)\in k\llbracket X\rrbracket$.
	For $f=X^{s_i}f_{i+1}$, the formula in Remark \ref{formula} then reads
	\[
	\left(\sum_{j=0}^{p-1} (1+X)^{j} (1+X^ph_j(X))\right) c_i\varphi(f_i)= X^{s_i}f_{i+1}.
	\]
	Putting $Y=X+1$, we have $\sum_{j=0}^{p-1}(1+X)^{j}=(Y^p-1)/(Y-1)=(Y-1)^{p-1}=X^{p-1}$. So the factor in front of $c_i \varphi(f_i)$ lies in $X^{p-1}+(X^p)$, which gives the required containment for $\varphi(f_i)$.
	
	Let now $\gamma\in \Gamma$. We have $\gamma(f_i)\in \pi_i^{\vee}$, say $\gamma(f_i)=h(X)f_i$ for some $h(X)\in k\llbracket X\rrbracket$. Evaluating at $v_i$ shows that $\chi_i(\gamma^{-1})=f_i(\gamma^{-1}(v_i))=h(0)$, so we are done.
\end{proof}

\begin{prop}\label{herzig_prop}
	Let $D\in \phigamma$ with basis $f_1,\ldots,f_n$ such that for all $1\leq i\leq n$,
	\begin{align*}
		&\varphi(f_i)\in d_i(1+Xk \llbracket X\rrbracket)X^{t_i}f_{i+1}\\
		&\gamma(f_i)\in \omega(\gamma)^{b_i}(1+Xk \llbracket X\rrbracket)f_i, \text{ for all $\gamma\in \Gamma$,}
	\end{align*}
	for some $b_i,t_i\in \bZ$ and $d_i\in k^{\times}$. Then $D$ is isomorphic to the $(\varphi,\Gamma)$-module with basis $e_1,\ldots,e_n$ such that
	\begin{align*}
		&\varphi(e_i)=e_{i+1} \text{ for $1\leq i \leq n-1$ and } \varphi(e_n)=d X^{-t(p-1)}e_1\\
		&\gamma(e_i)=\omega(\gamma)^{b_1} \left(\frac{\omega(\gamma)X}{\gamma(X)}\right)^{t\frac{p^{i-1}(p-1)}{p^{n}-1}} e_i \text{ for all $1\leq i\leq n$,}
	\end{align*}
	where\footnote{We will see in the proof that indeed $t\in \bZ$.} $t=-\frac{p^{n-1}t_1+p^{n-2}t_2+\ldots + t_n}{p-1}\in \bZ$ and $d=\prod_{i=1}^n d_i$. In particular, the associated Galois representation $\mathscr{V}(D)$ attached to $D$ via Fontaine's equivalence \ref{fontaine_equiv} is isomorphic to
	\[
	\mathscr{V}(D)\cong \omega^{b_1}\mu_{\lambda} \otimes \Ind^{G_{\bQ_p}}_{G_{\bQ_{p^n}}}(\omega_n^t),
	\]
	where $\lambda\in \bar{k}^{\times}$ with $\lambda^n=d$.
\end{prop}

\begin{proof}
	Write $\varphi(f_i)=d_i g_i(X)X^{t_i}f_{i+1}$ for some $g_i(X)\in 1+Xk\llbracket X\rrbracket$. We want to get rid of the $g_i(X)$ by replacing $f_i$ by $f_i'=h_i(X)f_i$ for suitable $h_i(X)\in 1+Xk\llbracket X\rrbracket$. The condition $\varphi(f_i')=d_i X^{t_i}f_{i+1}'$ is equivalent to $h_{i+1}(X)=h_i(X^p)g_i(X)$. This has the unique solution $h_i(X)=\prod_{j=1}^{\infty} g_{i-j}(X^{p^{j-1}})$ in $1+Xk\llbracket X\rrbracket$, see also the end of the proof of \cite[Proposition 3.2.4.2]{herzig_co}. Since the $f_i'$ still satisfy the same condition as the $f_i$, we may from now on assume that $\varphi(f_i)=d_i X^{t_i}f_{i+1}$ for all $1\leq i\leq n$.
	
	Put now $e_1:=f_1$ and $e_{i+1}:=\varphi(e_i)$ for $1\leq i \leq n-1$. Then $\varphi(e_n)=d X^{\sum} e_1$, where $d=\prod_{i=1}^n d_i$ and $\sum = \sum_{i=1}^{n} p^{n-i}t_i$. We still have $\gamma(e_i)\in \omega(\gamma)^{b_i}(1+Xk \llbracket X\rrbracket)e_i$ for all $1\leq i\leq n$ and $\gamma\in \Gamma$. Fix $\gamma\in \Gamma$ and write $\gamma(e_1)=\omega(\gamma)^{b_1} h_1(X) e_1$ for some $h_1(X)\in 1+Xk\llbracket X\rrbracket$. Since $\varphi$ and $\gamma$ commute, one obtains inductively that 
	\[
	\omega(\gamma)^{\Sigma}\left(\frac{\gamma(X)}{\omega(\gamma)X}\right)^{\Sigma} = h_1(X^{p^n})h_1(X)^{-1}.
	\]
	Since $h_1(X)$ is a $1$-unit, we deduce that $\omega(\gamma)^{\Sigma}=1$. This holds for every $\gamma\in \Gamma$, so we may write $\Sigma = -t(p-1)$ for some $t\in \bZ$. Now $\left(\frac{\gamma(X)}{\omega(\gamma)X}\right)^{\Sigma}$ has a unique $(p^n-1)$-th root in $1+X\bF_p\llbracket X\rrbracket$, say $g(X)$, and we deduce that $h(X^{p^n})=h(X)$ for $h(X)=h_1(X)g(X)^{-1}$, i.e.\ $h(X)$ is constant equal to $1$. This proves that $h_1(X)=\left(\frac{\omega(\gamma)X}{\gamma(X)}\right)^{t\frac{(p-1)}{p^{n}-1}}$ and thus the correct formula for $\gamma(e_1)$, which also implies the correct formulas for $\gamma(e_i)$, for all $1\leq i\leq n$.
	
	The description of the Galois representation $\mathscr{V}(D)$ now follows from Example \ref{example_phi_gamma}. Indeed, since $\omega_n^{p^n-1}=1$, we may replace $t$ by $t+p^{n}-1$ on the Galois side, which amounts to replacing the basis $\{e_i\}$ by $\{X^{-p^{i-1}(p-1)}e_i\}$ on the $(\varphi,\Gamma)$-module side. This way we may assume that $t\geq 0$, in which case the cited example applies.
\end{proof}

\begin{cor}\label{Galois_calculate}
	The Galois representation $\mathscr{V}(\pi_M^{\vee}[1/X])^{\vee}\in \Mod^{\fin}_{\GQp}(k)$ is isomorphic to
	\[
	\omega^{a_1-1}\mu_{\lambda} \otimes \Ind^{G_{\bQ_p}}_{G_{\bQ_{p^n}}}(\omega_n^s)
	\]
	where $\lambda\in \bar{k}$ such that $\lambda^n=\prod_{i=1}^n c_i$, $s=\frac{p^{n-1}s_1+p^{n-2}s_2+\ldots+s_n}{p-1}\in \bZ$ and $a_1\in \bZ$ is such that $\chi_1(\gamma)=\omega(\gamma)^{a_1}$ for all $\gamma\in \Gamma$. Moreover, if the condition formulated in part (iii) of Lemma \ref{basic_lemma} is satisfied, then this is an irreducible Galois representation.
\end{cor}

\begin{proof}
	Everything except the irreducibility statement follows from Lemma \ref{phi_Gamma_description} and Proposition \ref{herzig_prop}. Here note that we dualize the Galois representation and we have $\omega_n^{(p^n-1)/(p-1)}=\omega$. Assume now that the condition of Lemma \ref{basic_lemma} (iii) is satisfied, so that $\pi_M$ is an irreducible $P^+$-representation. Put $D=\pi_M^{\vee}[1/X]$. The canonical map $\pi_M^{\vee}\hookrightarrow D$ is compatible with the action of $F^{\vee}$ on the source and of $\psi$ on the target (as follows for example from Remark \ref{formula}). We then have $D^{\natural}\subset \pi_M^{\vee}$, which must be an equality by the irreducibility of $\pi_M$ as a $P^+$-representation. If now $D_1\subset D$ is a $(\varphi,\Gamma)$-submodule, then $D_1^{\natural}\subset D^{\natural}=\pi_M^{\vee}$ by functoriality of $(-)^{\natural}$, Remark \ref{sharp_functorial}, so once again by irreducibility $D_1^{\natural}=D^{\natural}$ and thus, by inverting $X$, $D_1=D$. This shows that the $(\varphi,\Gamma)$-module $D$ is irreducible.
\end{proof}

\subsection{The extended functor} For $\pi\in \Mod^{\fadm}_{\wt{G}}(k)$, we now check that the ind-object $\mathbf{D}(\pi)\in \Ind(\phigamma)$ is compact. Recall from Definition \ref{def_G(pi)} that to $\pi$ we have associated the filtered set $\mathcal{G}(\pi)$ consisting of the finite dimensional $\wt{K}\wt{Z}$-stable subspaces of $\pi$ generating the latter as a $\wt{G}$-representation. On the other hand, we may consider the filtered set $\mathcal{M}(\pi)$, Definition \ref{def_M(pi)}. Given some element $M\in \mathcal{G}(\pi)$, Proposition \ref{prop_finiteness} implies that the $k\llbracket X\rrbracket[F,\Gamma]$-submodule $\pi_M\subset \pi$ of $\pi$ generated by $M$ lies in $\mathcal{M}(\pi)$. The resulting map $\mathcal{G}(\pi)\to \mathcal{M}(\pi)$ preserves the partial order and thus induces a morphism
\begin{equation}\label{comp_iso}
	\underset{M\in \mathcal{G}(\pi)}{``{\varinjlim}"} (\pi_M^{\vee}[1/X])'\to \underset{\tau\in \mathcal{M}(\pi)}{``{\varinjlim}"} (\tau^{\vee}[1/X])'=\mathbf{D}(\pi)
\end{equation}
in the category $\Ind(\phigamma)$.

\begin{prop}\label{stabilize} Let $\tau_1,\tau_2\in \mathcal{M}(\pi)$ with $\tau_1\subset \tau_2$, and assume that $\tau_1[F^{-1}]=\pi$. The transition map $(\tau_1^{\vee}[1/X])'\xrightarrow{\cong} (\tau_2^{\vee}[1/X])'$ is an isomorphism. In particular $\mathbf{D}(\pi)=(\tau_1^{\vee}[1/X])'\in \phigamma$ is a compact object of $\Ind(\phigamma)$. Moreover, any finitely generated $k\llbracket X\rrbracket[F]$-submodule of $\pi$ has finite dimensional $X$-torsion.
\end{prop}

\begin{proof}
	This is similar to the proof of \cite[Proposition 4.4]{emerton_coherent_rings}: If $\tau_2$ is an arbitrary finitely generated $k\llbracket X\rrbracket[F]$-submodule of $\pi=\tau_1[F^{-1}]$ (not necessarily an element of $\mathcal{M}(\pi)$), then we may choose $N\gg 0$ such that $F^N\tau_2\subset \tau_1$. In particular, the injective operator $F^N$ maps $\tau_2[X]$ into $\tau_1[X^{p^N}]$ (which is finite dimensional since $\tau_1\in \mathcal{M}(\pi)$). This already proves the final assertion. For the first part, assume now that additionally $\tau_2\in \mathcal{M}(\pi)$ and $\tau_1\subset \tau_2$. Then the quotient $\tau_2/\tau_1$ is finitely generated over $k\llbracket X\rrbracket$ (as $\tau_2$ is finitely generated over $k\llbracket X\rrbracket[F]$ and the quotient is killed by $F^N$). But it is also torsion by smoothness, so it must be finite dimensional and is therefore killed when taking Pontryagin dual and inverting $X$, i.e.\ $\tau_1^{\vee}[1/X]=\tau_2^{\vee}[1/X]$, proving the first part.
\end{proof}

\begin{cor}\label{cor_compact}
	The morphism (\ref{comp_iso}) is an isomorphism, and all the transition maps in the inverse system $\{\pi_M^{\vee}[1/X]\}_{M\in \mathcal{G}(\pi)}$ are isomorphisms. In particular, $\mathbf{D}(\pi)\in \phigamma$ is a compact object of $\Ind(\phigamma)$. Moreover, any finitely generated $k\llbracket X\rrbracket [F]$-submodule of $\pi$ has finite dimensional $X$-torsion.
\end{cor}

\begin{proof}
	The decomposition $\wt{G}=P\wt{K}\wt{Z}$ implies that $\pi_M[F^{-1}]=\pi$ for $M\in \mathcal{G}(\pi)$.
\end{proof}

\begin{remark}
	The composition $\Mod^{\fadm}_{\wt{G}}(k)\xrightarrow{(-)|_P} \Mod^{\sm}_P(k)\xrightarrow{\mathbf{D}} \Ind(\phigamma)$ may thus be viewed as valued in the category $\phigamma$, and the resulting functor, still denoted by $\mathbf{D}$, recovers Colmez's functor when restricted to the full subcategory $\Mod^{\fadm}_G(k)$, see \cite[Proposition 3.2 (iii)]{Breuil_phi}.
\end{remark}

\section{Metaplectic Montr\'{e}al functor} In this section, we study the restriction of the functor $\mathbf{D}$ (or equivalently $\mathbf{V}$) to the full subcategory $\Mod^{\fadm}_{\wt{G},\iota}(k)$ of genuine objects. As we will soon see, the pullback of the metaplectic cover to the center $Z\cong \bQ_p^{\times}$ of $G$ defines an interesting action on this category. Translating this action over to the Galois side and passing to the category of equivariant objects, we obtain so-called metaplectic Galois representations and the previously extended functor induces the metaplectic Montr\'{e}al functor.

\subsection{Twist-equivariant $(\varphi,\Gamma)$-modules} 
Throughout this section we fix a group $\Omega$. We remind the reader what it means for $\Omega$ to act on a category and what the corresponding category of $\Omega$-equivariant objects is.

\begin{definition}[{\cite[2.1, 2.4]{shinder}}]\label{def_action_category} Let $\mathcal{C}$ be a catgory. 
	
	\begin{enumerate}
		\item[{\rm (i)}] A \textit{group action of $\Omega$ on $\mathcal{C}$} consists of the following data:
		\begin{itemize}
			\item for each $g\in \Omega$, an autoequivalence $A_g\colon \mathcal{C}\xrightarrow{\sim} \mathcal{C}$
			\item for each pair $g,h\in \Omega$, a natural isomorphism $\eta_{g,h}\colon A_gA_h \Rightarrow A_{gh}$
		\end{itemize}
		such that the diagram
		\[
		\xymatrix{
			A_gA_hA_k \ar@{=>}[r]^{A_g{\eta}_{h,k}} \ar@{=>}[d]_{\eta_{g,h}A_k} & A_g A_{hk}\ar@{=>}[d]^{\eta_{g,hk}}\\
			A_{gh}A_k \ar@{=>}[r]_{\eta_{gh,k}} & A_{ghk}
		}
		\]
		commutes for all elements $g,h,k\in \Omega$.
		\item[{\rm (ii)}] A \textit{$\Omega$-equivariant object} is a pair $(D, \{\theta_g \colon D\xrightarrow{\sim} A_g(D)\}_{g\in \Omega})$ consisting of an object $D$ of $\mathcal{C}$ and a collection of isomorphisms $\theta_g\colon D\xrightarrow{\sim} A_g(D)$, for $g\in \Omega$, such that the diagram
		\[
		\xymatrix{
			D\ar[r]^{\theta_g} \ar[d]_{\theta_{gh}} & A_g(D)\ar[d]^{A_g\theta_h}\\
			A_{gh}(D) & A_g(A_h(D))\ar[l]^{\eta_{g,h}(D)}
		}
		\]
		commutes for all $g,h\in \Omega$.
		
		The $\Omega$-equivariant objects form in an evident way a category $\mathcal{C}^{\Omega}$.
	\end{enumerate}
\end{definition}

Recall from Definition \ref{phi_gamma_twist} that for a $(\varphi,\Gamma)$-module $D$ and a smooth character $\chi\colon \bQ_p^{\times}\to k^{\times}$, we may form the twisted $(\varphi,\Gamma)$-module $D\otimes \chi$. This construction defines an action of the abelian group $\Hom^{\operatorname{cts}}(\bQ_p^{\times},k^{\times})$ on the category $\phigamma$ by taking $A_{\chi}=-\otimes \chi$ and $\eta_{\chi,\psi}=\operatorname{id}$ in Definition \ref{def_action_category}. In particular, given a group homomorphism
\[
\Omega\to \Hom^{\operatorname{cts}}(\bQ_p^{\times},k^{\times}), \hspace{0.1cm}g\mapsto \chi_g,
\]
we may consider the category $\left(\phigamma\right)^{\Omega}$ of $\Omega$-equivariant $(\varphi,\Gamma)$-modules. Its objects are pairs $(D,\theta)$ consisting of a $(\varphi,\Gamma)$-module $D\in \phigamma$ and an action $\theta\colon \Omega\to \Aut_{k((X))}(D)$ of $\Omega$ on the underlying $k((X))$-vector space of $D$ with $\theta(g)\colon D\cong D\otimes \chi_g$ being an isomorphism of $(\varphi,\Gamma)$-modules.

\subsubsection{Frobenius reciprocity}\label{section_Frobenius} Assume now that $\Delta\subset \Omega$ is a subgroup of finite index, which we assume to be normal for convenience. Consider the restriction functor
\begin{equation}\label{res_functor}
	(-)|_{\Delta}\colon \left(\phigamma\right)^{\Omega}\to \left(\phigamma\right)^{\Delta},\hspace{0.2cm} (D,\theta)\mapsto (D,\theta|_{\Delta}).
\end{equation}
We define a left adjoint by 
\[
\Ind^{\Omega}_{\Delta}((D,\theta))=\left\{f\colon \Omega\to D : f(hg)=\theta(h)f(g) \text{ for all $g\in \Omega$, $h\in \Delta$}\right\}
\]
with $k((X))$-vector space structure induced by that on $D$ and $\Omega$-action defined by $(g.f)(g')=f(g'g)$, i.e.\ so far this is simply the usual induction of the representation $\theta\colon \Delta\to \Aut_{k((X))}(D)$. We endow it with a $(\varphi,\Gamma)$-action by
\begin{itemize}
	\item $(\varphi.f)(g)=\chi_g(p)\varphi(f(g))$ for all $g\in \Omega$
	\item $(\gamma.f)(g)=\chi_g(\gamma)\gamma(f(g))$ for all $g\in \Omega$, for all $\gamma\in \Gamma$.
\end{itemize}

\begin{prop}
	\begin{enumerate}
		\item[{\rm (i)}] The $(\varphi,\Gamma)$-action on $\Ind^{\Omega}_{\Delta}((D,\theta))$ is well-defined and yields a $\Omega$-equivariant object in $\phigamma$.
		\item[{\rm (ii)}] The resulting functor 
		\[
		\Ind^{\Omega}_{\Delta}(-)\colon \left(\phigamma\right)^{\Delta}\to \left(\phigamma\right)^{\Omega}
		\]
		is a left adjoint of the restriction functor (\ref{res_functor}).
	\end{enumerate}
\end{prop}

\begin{proof}
	It is straight forward to check that the $(\varphi,\Gamma)$-action on $\Ind^{\Omega}_{\Delta}((D,\theta))$ is well-defined. Since $\Delta\subset \Omega$ is normal, we have
	\begin{equation}\label{Ind|}
		\Ind^{\Omega}_{\Delta}((D,\theta))|_{\Delta}\cong \bigoplus_{g\in \Delta\setminus \Omega} (D\otimes \chi_g,\theta^g),
	\end{equation}
	where $\theta^g\colon \Delta\xrightarrow{g(-)g^{-1}}\Delta\xrightarrow{\theta}\Aut_{k((X))}(D)=\Aut_{k((X))}(D\otimes \chi_g)$. This space is finite dimensional over $k((X))$ since $\Delta$ has finite index in $\Omega$. In particular, this is a $(\varphi,\Gamma)$-module, which finishes the proof of part (i). Part (ii) follows from the usual Frobenius reciprocity for representations of $\Delta$ and $\Omega$, respectively.
\end{proof}

\begin{remark}\label{remark_action_ind}
	The action of $\Hom^{\operatorname{cts}}(\bQ_p^{\times},k^{\times})$ (hence of $\Omega$) on $\phigamma$ extends to the ind-category $\Ind(\phigamma)$. 
	The functors $\Ind^{\Omega}_{\Delta}(-)$ and $(-)|_{\Delta}$ extend correspondingly and the adjunction of part (ii) in the previous proposition carries over.
\end{remark}

\subsubsection{Transferring twist-actions with a warning}\label{section_transferring}
Using Fontaine's equivalence \ref{fontaine_equiv}, we transfer the action of $\Omega$ on $\phigamma$ to the category $\Mod^{\fin}_{\GQp}(k)$, which then induces an equivalence 
\begin{equation}\label{transfer}
	\left(\phigamma\right)^{\Omega}\simeq \left(\Mod^{\fin}_{\GQp}(k)\right)^{\Omega}
\end{equation}
on $\Omega$-equivariant objects. 

We warn the reader that the transferred action of $\Omega$ on $\Mod^{\fin}_{\GQp}(k)$ is not simply given by twisting by the characters $\chi_g$, $g\in \Omega$, viewed as characters of $\GQp$ via local class field theory: For $D\in \phigamma$ and $\chi\colon \bQ_p^{\times}\to k^{\times}$ is a smooth character, there is an isomorphism $\mathscr{V}(D\otimes \chi)\cong \mathscr{V}(D)\otimes \chi$, it however depends on the choice of a $(p-1)$-th root of $\chi(p)$ and is therefore not canonical. In general, one will not be able to choose the $(p-1)$-th roots of the elements $\chi_g(p)$ for $g\in \Omega$ in a multiplicatively compatible way.

\subsection{Metaplectic Galois representations} We now take $\Omega=\wt{Z}\cong \wt{\bQ_p^{\times}}$ to be the pullback of the metaplectic cover to the center $Z\cong \bQ_p^{\times}$ of $G$:
\[
1\to \mu_2\to \wt{\bQ_p^{\times}}\to \bQ_p^{\times}\to 1,
\]
and we consider the surjective group homomorphism
\begin{align*}
	\wt{\bQ_p^{\times}}\twoheadrightarrow \bQ_p^{\times}&\twoheadrightarrow \Hom^{\operatorname{cts}}(\bQ_p^{\times},\mu_2)\\
	z&\mapsto \chi_z=(\mu_{-1}^{v(z)}\omega^{v(z)}\mu_{\omega(z)})^{\frac{p-1}{2}},
\end{align*}
where $v\colon \bQ_p^{\times}\to \bZ$ is the $p$-adic valuation. We write $\chi_{\tilde{z}}=\chi_z$ for $\tilde{z}\in \wt{\bQ_p^{\times}}$ with image $z\in \bQ_p^{\times}$. The definition of these characters comes from the following lemma, which is quickly verified using the explicit cocycle (\ref{cocycle}), see also \cite[Lemma 4.4]{witthaus_meta}.

\begin{lem}\label{Ztilde_action}
	The conjugation action $\tilde{z}(-)(\tilde{z})^{-1}$ of an element $\tilde{z}\in \wt{Z}\cong \wt{\bQ_p^{\times}}$ on $\wt{G}$ is given by the $\wt{G}\supset \mu_2$-valued homomorphism 
	$\chi_{\tilde{z}}\circ {\det}$.
	
	In particular, for $\pi\in \Mod^{\sm}_{\wt{G},\iota}(k)$ (resp.\ $\Mod^{\sm}_{\wt{B},\iota}(k)$) and $\tilde{z}\in \wt{\bQ_p^{\times}}\cong \wt{Z}$, the action map $\tilde{z}\colon \pi \xrightarrow{\cong} \pi\otimes \chi_{\tilde{z}}\circ {\det}$ is $\wt{G}$-equivariant (resp.\ $\wt{B}$-equivariant) .
\end{lem}

Hence, if we let $\wt{\bQ_p^{\times}}$ act on $\Mod^{\sm}_{\wt{B},\iota}(k)$ via twisting by the characters $\chi_{\tilde{z}}\circ \det $, then we have the canonical functor
\begin{equation}\label{can_functor_B}
	\Mod^{\sm}_{\wt{B},\iota}(k)\to \left(\Mod^{\sm}_{\wt{B},\iota}(k)\right)^{\wt{\bQ_p^{\times}}}, \hspace{0.1cm} \pi\mapsto (\pi,\pi|_{\wt{\bQ_p^{\times}}}),
\end{equation}
the image of which lies in the full subcategory of $\wt{\bQ_p^{\times}}$-equivariant objects whose underlying $\wt{\bQ_p^{\times}}$-action is smooth and genuine. This motivates the following definition:

\begin{definition}\label{def_metaplectic_category}
	\begin{enumerate}
		\item[{\rm (i)}] A \textit{metaplectic $(\varphi,\Gamma)$-module over $k((X))$} is a $\wt{\bQ_p^{\times}}$-equivariant $(\varphi,\Gamma)$-module $(D,\theta)$ over $k((X))$ such that $\theta\colon \wt{\bQ_p^{\times}}\to \Aut_{k((X))}(D)$ is smooth and genuine. We denote the resulting category by $\mphigamma$.
		\item[{\rm (ii)}] The \textit{category of metaplectic Galois representations of $\GQp$ on $k$-vector spaces} is the category $\wt{\Mod}^{\fin}_{G_{\bQ_p},\iota}(k)$ equivalent to $\mphigamma$ under (\ref{transfer}).
	\end{enumerate}
\end{definition}

As explained in Remark \ref{remark_twist_Br}, the functor $\mathbf{D}\colon \Mod^{\sm}_{\wt{B},\iota}(k)\xrightarrow{(-)|_P} \Mod^{\sm}_P(k)\to \Ind(\phigamma)$ is compatible with twists and thus induces a functor on $\wt{\bQ_p^{\times}}$-equivariant objects. Precomposing with (\ref{can_functor_B}) gives rise to the horizontal arrows in the left square of the diagram
\begin{equation}\label{big_diag}
	\xymatrix{
		\Mod^{\sm}_{\wt{B},\iota}(k) \ar[r]^{\wt{\mathbf{D}}} \ar@/^1.5pc/[rr]^{\mathbf{D}}& \Ind(\phigamma)^{\wt{\bQ_p^{\times}}} \ar[r]^{\mathrm{forget}} & \Ind(\phigamma)\\
		\Mod^{\sm}_{\wt{B},\iota}(k)_{Z(\wt{G})-\mathrm{fin}} \ar@{^{(}->}[u]\ar[r]^{\wt{\mathbf{D}}} & \Ind(\mphigamma)\ar@{^{(}->}[u] \ar[r]^{\mathrm{forget}} & \Ind(\phigamma) \ar@{=}[u]\\
		\Mod^{\fadm}_{\wt{G},\iota}(k) \ar[u]^{(-)|_{\wt{B}}} \ar[r]^{\wt{\mathbf{D}}}& \mphigamma \ar@{^{(}->}[u]\ar[r]^{\mathrm{forget}}  & \phigamma, \ar@{^{(}->}[u]
	}
\end{equation}
where for the middle horizontal functor we implicitly made use of Lemma \ref{cofinal_Z_stable}.

\begin{definition}
	The \textit{metaplectic Montr\'{e}al functor} is the functor
	\[
	\wt{\mathbf{D}}\colon \Mod^{\fadm}_{\wt{G},\iota}(k) \to \mphigamma
	\]
	appearing at the bottom of diagram (\ref{big_diag}). We further define
	\[
	\wt{\mathbf{V}}=\mathscr{V}\circ \wt{\mathbf{D}}\colon \Mod^{\fadm}_{\wt{G},\iota}(k)\to \wt{\Mod}^{\fin}_{G_{\bQ_p},\iota}(k)
	\]
	to be the composition of $\wt{\mathbf{D}}$ with Fontaine's equivalence, and we also refer to it as the metaplectic Montr\'{e}al functor.
\end{definition}

\begin{remark}\label{Dtilde_twist}
	The functor $\wt{\mathbf{D}}\colon \Mod^{\sm}_{\wt{B},\iota}(k)\to \Ind(\phigamma)^{\wt{\bQ_p^{\times}}}$ is compatible with twists: if $\wt{\mathbf{D}}(\pi)=(\mathbf{D}(\pi),\theta)$ and $\chi\colon \bQ_p^{\times}\to k^{\times}$ is a smooth character, then 
	\[
	\wt{\mathbf{D}}(\pi\otimes \chi\circ \det)=(\mathbf{D}(\pi)\otimes \chi, \theta\otimes \chi^2),
	\]
	where $(\theta\otimes \chi^2)_{\tilde{z}}=\chi(z^2)\theta_{\tilde{z}}$ for $\tilde{z}\in \wt{\bQ_p^{\times}}$ with image $z\in \bQ_p^{\times}$.
\end{remark}

Since the character $\chi_{\tilde{z}}$ is trivial for $\tilde{z}\in \wt{S}=S\times \mu_2$, Frobenius reciprocity for the inclusion $S\subset \wt{S}\subset \wt{\bQ_p^{\times}}$ as presented in Section \ref{section_Frobenius} (restricted to smooth genuine objects) may be interpreted as an adjoint pair
\begin{equation}\label{frob_rec}
	\Ind^{\wt{\bQ_p^{\times}}}_{\wt{S}}(-\boxtimes \iota)\colon \Mod^{\sm}_S(\phigamma) \rightleftarrows \mphigamma \noloc (-)|_S,
\end{equation}
where the source category consists of group homomorphisms $S\to \Aut_{\phigamma}(D)$, for $D\in \phigamma$, such that the underlying $k((X))$-linear representation of $S$ on $D$ is smooth. A corresponding version exists on the level of ind-categories, cf.\ Remark \ref{remark_action_ind}.\\

Recall from Lemma \ref{PS} (i) that $\wt{B}_S=B_S\times \mu_2$ as groups, so the functor $-\boxtimes \iota$ gives an equivalence of categories $\Mod^{\sm}_{B_S}(k)\simeq \Mod^{\sm}_{\wt{B}_S,\iota}(k)$.

\begin{prop}\label{D_comp_Ind}
	Let $\pi\in \Mod^{\sm}_B(k)$ with central character $\chi\colon \bQ_p^{\times}\to k^{\times}$. Then there is a canonical isomorphism
	\[
	\wt{\mathbf{D}}(\Ind^{\wt{B}}_{\wt{B}_S}(\pi|_{B_S}\boxtimes \iota))\cong \Ind^{\wt{\bQ_p^{\times}}}_{\wt{S}}(\chi|_S\boxtimes \mathbf{D}(\pi)\boxtimes \iota)
	\]
	in the category $\Ind(\mphigamma)$, where $\chi|_S\boxtimes \mathbf{D}(\pi)\in \Mod^{\sm}_S(\phigamma)$ is the action of $S$ on $\mathbf{D}(\pi)$ via $\chi|_S$.
\end{prop}

\begin{proof}
	The inclusion $\wt{\bQ_p^{\times}}\cong \wt{Z}\subset \wt{B}$ induces a bijection $\wt{S}\setminus \wt{\bQ_p^{\times}}\cong \wt{B}_S\setminus \wt{B}$ and it follows from Lemma \ref{Ztilde_action} that
	\[
	\Ind^{\wt{B}}_{\wt{B}_S}(\pi|_{B_S}\boxtimes \iota)|_{B_S}\cong \bigoplus_{\tilde{z}\in \wt{S}\setminus \wt{\bQ_p^{\times}}} (\pi\otimes \chi_{\tilde{z}}\circ \det)|_{B_S}.
	\]
	Since the definition of $\mathbf{D}$ only depends on the restriction to $P$, we deduce that
	\begin{equation}\label{D(Ind)}
		\mathbf{D}(\Ind^{\wt{B}}_{\wt{B}_S}(\pi|_{B_S}\boxtimes \iota))\cong \bigoplus_{\tilde{z}\in \wt{S}\setminus \wt{\bQ_p^{\times}}} \mathbf{D}(\pi)\otimes \chi_{\tilde{z}}.
	\end{equation}
	This contains $\mathbf{D}(\pi)$ and, since $S$ acts on $\wt{\mathbf{D}}(\Ind^{\wt{B}}_{\wt{B}_S}(\pi|_{B_S}\boxtimes \iota))$ via the character $\chi|_S$, we may apply (the ind-version of) Frobenius reciprocity (\ref{frob_rec}) to obtain a non-trivial morphism
	\[
	\Ind^{\wt{\bQ_p^{\times}}}_{\wt{S}}(\chi|_S\boxtimes \mathbf{D}(\pi)\boxtimes \iota) \to \wt{\mathbf{D}}(\Ind^{\wt{B}}_{\wt{B}_S}(\pi|_{B_S}\boxtimes \iota))
	\]
	in the category $\Ind(\mphigamma)$. That this is an isomorphism can be checked in the category $\Ind(\phigamma)$, where it follows from (\ref{D(Ind)}) and (\ref{Ind|}).
\end{proof}

\subsection{From metaplectic Galois to metaplectic Borel}

Following Colmez's recipe in \cite[Chapter III]{colmez1}, we explain how to attach a smooth genuine representation of $\wt{B}$ to a metaplectic $(\varphi,\Gamma)$-module.\\

For a $(\varphi,\Gamma)$-module $D\in \phigamma$, we have the operator $\psi$ as well as the largest $(\psi,\Gamma)$-stable lattice $D^{\sharp}\subset D$ as introduced in Section \ref{section_psi-gamma}. The operator $\psi$ on $D^{\sharp}$ is in general not invertible, but this can be overcome by passing to the limit
\[
\psi^{-\infty}(D):=\varprojlim_{\psi} D^{\sharp}=\{(v_n)_{n\geq 0} : \psi(v_{n+1})=v_n, \forall n\geq 0\}.
\]

\begin{remark}
	The canonical map
	\[
	\psi^{-\infty}(D)\xrightarrow{\cong} (\varprojlim_{\psi} D)_{b},
	\]
	where the right-hand side denotes bounded $\psi$-compatible sequences, is bijective. Indeed, $D^{\sharp}$ is finitely generated over $k\llbracket X\rrbracket$, hence bounded, i.e.\ the map is well-defined. Conversely, if $(v_n)_{n\geq 0}$ is a bounded $\psi$-compatible sequence, then there exists some $d\geq 0$ such that $v_n\in X^{-p^d}D^{\sharp}$. Since the inclusion $D^{\sharp}\subset X^{-p^d}D^{\sharp}$ has finite dimensional cokernel, the defining property of $D^{\sharp}$ allows us to choose $N\geq 0$ with $\psi^{N}(X^{-p^d}D^{\sharp})\subset D^{\sharp}$. In particular, $v_n=\psi^{N}(v_{n+N})\in D^{\sharp}$ for all $n\geq 0$.
\end{remark}

\begin{prop}[{\cite[Proposition III.1.1]{colmez1}}]\label{psi_P_action} The inverse limit $\psi^{-\infty}(D)=\varprojlim_{\psi} D^{\sharp}$ is a $k\llbracket X\rrbracket$-module via $f(X)\cdot (v_n)_n=(\varphi^n(f(X))v_n)_n$ and the following identities uniquely determine an action of the mirabolic subgroup $P$ on it: For $v=(v_n)_{n\geq 0}\in \psi^{-\infty}(D)$,
	\begin{align*}
		&\begin{mat} p^m & 0\\0 & 1\end{mat}v=(v_{n+m})_{n\geq 0} \text{ for all $m\in \bZ$,}\\
		&\begin{mat}
			\gamma & 0\\0 & 1
		\end{mat}v=(\gamma(v_n))_{n\geq 0} \text{ for all $\gamma\in \bZ_p^{\times}=\Gamma$,}\\
		&\begin{mat}
			1 & b\\0 & 1
		\end{mat}v=(X+1)^{-b}v=((X^{p^n}+1)^{-b}v_n)_{n\geq 0} \text{ for all $b\in \bZ_p$,}
	\end{align*}
	where we use the convention that $v_n=\psi^{-n}(v_0)$ if $n<0$.
\end{prop}

\begin{remark}
	Our definition of the action of $P$ on $\psi^{-\infty}(D)$ differs from the one of Colmez by conjugation by the element $\begin{mat}
		-1 & 0\\0 & 1
	\end{mat}$.
\end{remark}

\begin{lem}\label{psi_sharp}
	The assignment $D\mapsto \psi^{-\infty}(D)$ defines an exact functor from $\phigamma$ to the category of $k$-linear representations of $P$, which is compatible with twists: If $\chi\colon \bQ_p^{\times}\to k^{\times}$ is a smooth character, then the map
	\begin{align*}
		\varprojlim_{\psi} (D\otimes \chi)^{\sharp} &\cong (\varprojlim_{\psi} D^{\sharp}) \otimes \chi \circ {\det}\\
		(v_n)_{n\geq 0} & \mapsto (\chi(p)^{-n}v_n)_{n\geq 0}
	\end{align*}
	is a $P$-equivariant isomorphism. 
\end{lem}

\begin{proof}
	While functoriality follows from the assignment $D\mapsto D^{\sharp}$ being functorial, see Remark \ref{sharp_functorial}, exactness is the content of \cite[Theorem III.3.5]{colmez1}.
	
	Let now $\chi\colon \bQ_p^{\times}\to k^{\times}$ be a smooth character. It follows from the definition of $D\otimes \chi$,  Definition \ref{phi_gamma_twist}, that $(D\otimes \chi)^{\sharp}=D^{\sharp}\otimes \chi$ with $\psi_{D\otimes \chi}(v)=\chi(p)^{-1}\psi_D(v)$ for all $v\in D$. From this one verifies that the described map is a well-defined $P$-equivariant isomorphism.
\end{proof}

We endow $D^{\sharp}$ with the $X$-adic topology and give $\psi^{-\infty}(D)=\varprojlim_{\psi} D^{\sharp}$ the inverse limit topology. Here note that the transition maps are $k\llbracket X\rrbracket$-linear if the $n$-th component is viewed as an $k\llbracket X\rrbracket$-module via the $n$-th power $\varphi^n$  (this defines the same topology as the $X$-adic one). Hence, $\varprojlim_{\psi} D^{\sharp}$ is a profinite $k\llbracket X\rrbracket$-module. Continuity of the $\Gamma$-action on $D$ implies that the action map $P\times \psi^{-\infty}(D)\to \psi^{-\infty}(D)$ is continuous (for the product topology on the source), hence the contragredient action of $P$ on the Pontryagin dual $(\psi^{-\infty}(D))^{\vee}$ is smooth. Precomposing the functor $D\mapsto \psi^{-\infty}(D)$ of Lemma \ref{psi_sharp} with the dual $(-)'$ of $(\varphi,\Gamma)$-modules and postcomposing with the Pontryagin dual, we thus obtain the exact covariant functor
\begin{equation}\label{functor_Omega}
	\bOmega:=(\psi^{-\infty}\circ (-)')^{\vee}\colon \phigamma \to \Mod^{\sm}_P(k),
\end{equation}
whose notation we adopt from \cite{berger_onsome}. By Lemma \ref{Ztilde_action}, we have the semi-direct product decomposition
\begin{equation}\label{Btilde_semidirect}
	\wt{B}=\wt{Z}\rtimes P,
\end{equation}
where the conjugation action of $P$ on $\wt{Z}$ is given by the group homomorphism $P\to \Aut(\wt{Z})$, $g\mapsto [\tilde{z}\mapsto \chi_{\tilde{z}}(\det(g))]$. 

Let now $(D,\theta)\in \mphigamma$ be a metaplectic $(\varphi,\Gamma)$-module. Besides the smooth $k$-representation $\bOmega(D)$ of $P$, we then have, for $\tilde{z}\in \wt{\bQ_p^{\times}}\cong \wt{Z}$, the $P$-equivariant isomorphism $\bOmega(\theta_{\tilde{z}})\colon \bOmega(D)\xrightarrow{\cong} \bOmega(D\otimes \chi_{\tilde{z}})\cong \bOmega(D)\otimes \chi_{\tilde{z}}\circ \det$ (implicitly making use of the isomorphism in Lemma \ref{psi_sharp}), which assemble into a smooth and genuine action of $\wt{Z}$ on the underlying $k$-vector space of $\bOmega(D)$. The semi-direct product decomposition (\ref{Btilde_semidirect}) and the $P$-equivariance of the isomorphisms $\bOmega(\theta_{\tilde{z}})$ allows us to uniquely glue the $\wt{Z}$-action on $\bOmega(D)$ to a smooth and genuine $k$-representation $\wt{\bOmega}(D)$ of $\wt{B}$. This construction gives rise to an exact covariant functor 
\[
\wt{\bOmega}\colon \mphigamma \to \Mod^{\sm}_{\wt{B},\iota}
\]
making the diagram
\[
\xymatrix{
	\phigamma \ar[r]^{\bOmega} & \Mod^{\sm}_P(k)\\
	\mphigamma \ar[u]^{\mathrm{forget}} \ar[r]^{\wt{\bOmega}} & \Mod^{\sm}_{\wt{B},\iota}(k)\ar[u]_{(-)|_P}
}
\]
commute.

\begin{prop}\label{can_B_equiv_map}
	For $\pi \in \Mod^{\fadm}_{\wt{G},\iota}(k)$, there is a canonical $\wt{B}$-equivariant map
	\[
	\wt{\bOmega}(\wt{\mathbf{D}}(\pi))\to \pi.
	\]
	If $\wt{\mathbf{D}}(\pi)\neq 0$, then this map is non-zero.
\end{prop}

\begin{proof}
	We construct the Pontryagin dual of the desired map. Let $M\in \mathcal{G}(\pi)$, so that $\pi=\pi_M[F^{-1}]$ and $\mathbf{D}(\pi)'=\pi_M^{\vee}[1/X]$ by Corollary \ref{cor_compact}. The formula in Remark \ref{formula} implies that the canonical map $\pi^{\vee}\to \pi_M^{\vee}\to \pi_M^{\vee}[1/X]=\mathbf{D}(\pi)'$ is compatible with the map $F^{\vee}$ on the source and $\psi$ on the target. Since $F$ is injective on $\pi_M$, the dual is surjective, hence the image of the bounded $k\llbracket X\rrbracket$-module $\pi_M^{\vee}$ in $\pi_M^{\vee}[1/X]$ is contained in $(\pi_M^{\vee}[1/X])^{\sharp}$, so we may consider the map
	\[
	\pi^{\vee}\cong \varprojlim_{F^{\vee}} \pi^{\vee}\to \varprojlim_{\psi}(\pi_M^{\vee}[1/X])^{\sharp}=\psi^{-\infty}(\mathbf{D}(\pi)')=\wt{\bOmega}(\wt{\mathbf{D}}(\pi))^{\vee},
	\]
	which is non-trivial if $\mathbf{D}(\pi)'=\pi_M^{\vee}[1/X]\neq 0$. Here, the first isomorphism is given by $f\mapsto \left(f\circ \begin{mat}
		p^{-n} & 0\\0 & 1
	\end{mat}\right)_{n\geq 0}$, from which one quickly verifies the $P$-equivariance. The $\wt{Z}$-equivariance follows from the construction.
\end{proof}

We finish this section with a useful compatibility of the functor $\wt{\bOmega}$ with the induction functor (\ref{frob_rec}), which is a counterpart of Proposition \ref{D_comp_Ind}.

\begin{prop}\label{Psi_Ind}
	For $D\in \phigamma$ and $\chi\colon S\to k^{\times}$ a smooth character, there is a canonical $\wt{B}$-equivariant isomorphism
	\[
	\Ind^{\wt{B}}_{\wt{B}_S}(\chi\boxtimes \bOmega(D)\boxtimes \iota)\cong \wt{\bOmega}\left(\Ind^{\wt{\bQ_p^{\times}}}_{\wt{S}}(\chi \boxtimes D\boxtimes \iota)\right),
	\]
	where $\chi\boxtimes \bOmega(D)$ is viewed as a representation of $B_S$ via the identification $B_S\cong S\times P$ (with $S$ embedded diagonally) and $\chi\boxtimes D\in \Mod^{\sm}_S(\phigamma)$ is given by the action of $S$ on $D$ via $\chi$.
\end{prop}

\begin{proof}
	By equation (\ref{Ind|}) applied to $\wt{S}\subset \wt{\bQ_p^{\times}}$ (which is a central subgroup), we have the isomorphism
	\[
	\Ind^{\wt{\bQ_p^{\times}}}_{\wt{S}}(\chi \boxtimes D\boxtimes \iota)|_S\cong \bigoplus_{\wt{z}\in \wt{S}\setminus \wt{\bQ_p^{\times}}} \chi\boxtimes (D\otimes \chi_{\tilde{z}})
	\]
	in the category $\Mod^{\sm}_S(\phigamma)$. The right-hand side contains $\chi\boxtimes D$ as a subobject. Applying the functor $\bOmega$ and using Frobenius reciprocity ($S$ acts via the character $\chi$ on both sides), we obtain a $\wt{B}$-equivariant map
	\[
	\Ind^{\wt{B}}_{\wt{B}_S}(\chi\boxtimes \bOmega(D)\boxtimes \iota)\to \wt{\bOmega}\left(\Ind^{\wt{\bQ_p^{\times}}}_{\wt{S}}(\chi \boxtimes D\boxtimes \iota)\right),
	\]
	which is an isomorphism as can be seen after restricting to $P$, where both sides are identified with $\oplus_{\tilde{z}\in \wt{S}\setminus \wt{\bQ_p^{\times}}} \bOmega(D)\otimes \chi_{\tilde{z}}$, see also the proof of Proposition \ref{D_comp_Ind}.
\end{proof}

\subsection{Images of absolutely irreducible genuine representations} In this final section, we compute the images of the absolutely irreducible genuine objects under the metaplectic Montr\'{e}al functor $\wt{\mathbf{D}}\colon \Mod^{\fadm}_{\wt{G},\iota}(k)\to \mphigamma$, or equivalently under $\wt{\mathbf{V}}\colon \Mod^{\fadm}_{\wt{G},\iota}(k)\to \wt{\Mod}^{\fin}_{G_{\bQ_p},\iota}(k)$. The results in Section \ref{section_recollection} imply that an absolutely irreducible genuine $k$-representation of $\wt{G}$ is either a principal series or supersingular.

\subsubsection{Genuine principal series} For a pair of smooth characters $\chi_1,\chi_2\colon \bQ_p^{\times}\to k^{\times}$, define
\[
\tilde{\pi}(\chi_1,\chi_2)=\Ind^{\wt{G}}_{\wt{B}_S}(\chi_1\otimes \chi_2|_S \boxtimes \iota)=\Ind^{\wt{G}}_{\wt{B}}(\Ind^{\wt{T}}_{\wt{T}_S}(\chi_1\otimes \chi_2|_S\boxtimes \iota)),
\]
whose isomorphism class only depends on the restriction $(\chi_1|_S,\chi_2|_S)$, see Proposition \ref{PS}. Evaluation at the identity yields the short exact sequence
\begin{equation}\label{ses_PS}
	0\to \tilde{\pi}(\chi_1,\chi_2)^{\operatorname{bc}}\to \tilde{\pi}(\chi_1,\chi_2)\to \Ind^{\wt{T}}_{\wt{T}_S}(\chi_1\otimes \chi_2|_S\boxtimes \iota)\to 0
\end{equation}
of $\wt{B}$-representations, where the kernel is the subspace of functions whose support is contained in the big cell, i.e.\ in the preimage of $B\begin{mat}
	0 & 1\\1 & 0
\end{mat}B$ under the projection $\wt{G}\twoheadrightarrow G$. By the proposition just cited and \cite[Proposition 4.28]{witthaus_meta}, both cokernel and kernel are irreducible. Note that any splitting of the above sequence would yield a finite dimensional $\wt{G}$-stable subspace of the irreducible $\wt{G}$-representation $\tilde{\pi}(\chi_1,\chi_2)$, which is infinite dimensional. Hence, the sequence is non-split.\\

We define $\Ind^G_B(\chi_1\otimes \chi_2)=\pi(\chi_1,\chi_2)\supset \pi(\chi_1,\chi_2)^{\operatorname{bc}}$ similarly.

\begin{lem}\label{bc_induced}
	As $\wt{B}$-representations, $\tilde{\pi}(\chi_1,\chi_2)^{\operatorname{bc}}\cong \Ind^{\wt{B}}_{\wt{B}_S}(\pi(\chi_1,\chi_2)^{\operatorname{bc}}|_{B_S}\boxtimes \iota)$.
\end{lem}

\begin{proof}
	The assignment $f\mapsto \left[x\mapsto f\left(\begin{mat}
		0 & 1\\1 & x
	\end{mat}\right)\right]$ defines a $B$-equivariant isomorphism from $\pi(\chi_1,\chi_2)^{\operatorname{bc}}$ onto $\mathscr{C}^{\infty}_c(\bQ_p,k)\otimes (\chi_2\otimes \chi_1)$ for an appropriate action of $B$ on the space $\mathscr{C}^{\infty}_c(\bQ_p,k)$ of locally constant compactly supported functions. Similarly, we may $\wt{B}$-equivariantly identify $\tilde{\pi}(\chi_1,\chi_2)^{\operatorname{bc}}\cong \mathscr{C}^{\infty}_c(\bQ_p,k)\otimes \Ind^{\wt{T}}_{\wt{T}_S}(\chi_2\otimes \chi_1|_S\boxtimes \iota)$, see \cite[Lemma 4.23]{witthaus_meta} for more details. The projection formula now gives the assertion
\end{proof}

\begin{prop}\label{image_ps} Let $\chi_1,\chi_2\colon \bQ_p^{\times}\to k^{\times}$ be smooth characters.
	\begin{enumerate}
		\item[{\rm (i)}] The object $\Ind^{\wt{\bQ_p^{\times}}}_{\wt{S}}((\chi_1\chi_2)|_S \boxtimes \mathscr{D}(\chi_2)\boxtimes \iota)$ is irreducible in $\mphigamma$, where $\mathscr{D}$ is Fontaine's equivalence \ref{fontaine_equiv}.
		\item[{\rm (ii)}] There is an isomorphism
		\[
		\wt{\mathbf{D}}(\tilde{\pi}(\chi_1,\chi_2)) \cong \Ind^{\wt{\bQ_p^{\times}}}_{\wt{S}}((\chi_1\chi_2)|_S \boxtimes \mathscr{D}(\chi_2)\boxtimes \iota)
		\]
		in the category $\mphigamma$.
		\item[{\rm (iii)}] The canonical $\wt{B}$-equivariant map $\wt{\bOmega}\wt{\mathbf{D}}(\tilde{\pi}(\chi_1,\chi_2))\to \tilde{\pi}(\chi_1,\chi_2)$ of Proposition \ref{can_B_equiv_map} has image $\tilde{\pi}(\chi_1,\chi_2)^{\operatorname{bc}}$ and induces a non-split short exact sequence
		\[
		0\to \Ind^{\wt{T}}_{\wt{T}_S}(\chi_2\omega \otimes (\chi_1 \omega^{-1})|_S\boxtimes \iota)\to \wt{\bOmega}\wt{\mathbf{D}}(\tilde{\pi}(\chi_1,\chi_2))\to \tilde{\pi}(\chi_1,\chi_2)^{\operatorname{bc}}\to 0.
		\]
	\end{enumerate}
\end{prop}

\begin{proof}
	By equation (\ref{Ind|}), we have
	\begin{equation}\label{res_to_S}
		\Ind^{\wt{\bQ_p^{\times}}}_{\wt{S}}((\chi_1\chi_2)|_S \boxtimes \mathscr{D}(\chi_2)\boxtimes \iota)|_{S}= \bigoplus_{\tilde{z}\in \wt{S}\setminus \wt{\bQ_p^{\times}}} (\chi_1\chi_2)|_S \boxtimes (\mathscr{D}(\chi_2)\otimes \chi_{\tilde{z}})
	\end{equation}
	in the category $\Mod^{\sm}_S(\phigamma)$.
	Since the summands are pairwise non-isomorphic, we deduce irreduciblity proving part (i).
	
	For (ii), note that the cokernel of (\ref{ses_PS}) is finite dimensional and therefore is $k[F]$-torsion. By Proposition \ref{exactness_properties} (ii), we thus have $\wt{\mathbf{D}}(\tilde{\pi}(\chi_1,\chi_2))=\wt{\mathbf{D}}(\tilde{\pi}(\chi_1,\chi_2)^{\operatorname{bc}})$. By the same reasoning, we have $\mathbf{D}(\pi(\chi_1,\chi_2)^{\operatorname{bc}})=\mathbf{D}(\pi(\chi_1,\chi_2))$, which is known to be equal to $\mathscr{D}(\chi_2)$, see \cite[Th\'{e}or\`{e}me 0.10 (iii)]{colmez2} (beware of Colmez's normalization). Part (ii) now follows by applying the functor $\wt{\mathbf{D}}$ to the isomorphism in Lemma \ref{bc_induced} and utilizing Proposition \ref{D_comp_Ind}.
	
	For (iii), note that an analogue for $\pi(\chi_1,\chi_2)$ was proved by Berger \cite[Proposition 1.2.4, Theoreme 2.1.1]{berger_french}, namely there is a non-split short exact sequence
	\[
	0\to \chi_2\omega\otimes \chi_1\omega^{-1}\to \bOmega (\mathscr{D}(\chi_2))\to \pi(\chi_1,\chi_2)^{\operatorname{bc}}\to 0
	\]
	of $P$-representations, in fact of $B=Z\times P$-representation if we let $Z\cong \bQ_p^{\times}$ act on the middle term via $\chi_1\chi_2$. Applying the functor $\Ind^{\wt{B}}_{\wt{B}_S}((\chi_1\chi_2)|_S\boxtimes (-) \boxtimes \iota)$, using Proposition \ref{Psi_Ind} and part (ii), we obtain a short exact sequence
	\[
	0\to \Ind^{\wt{T}}_{\wt{T}_S}(\chi_2\omega \otimes (\chi_1 \omega^{-1})|_S\boxtimes \iota)\to \wt{\bOmega}\wt{\mathbf{D}}(\tilde{\pi}(\chi_1,\chi_2))\to \tilde{\pi}(\chi_1,\chi_2)^{\operatorname{bc}}\to 0.
	\]
	
	Comparing this to the non-split sequence (\ref{ses_PS}) via the canonical map of interest, which is non-zero by Proposition \ref{can_B_equiv_map}, implies part (iii) except non-splitness, which however follows by observing that any splitting would dually give a non-zero map from the dual of the finite dimensional kernel to the $X$-torsionfree space $(\wt{\bOmega}\wt{\mathbf{D}}(\tilde{\pi}(\chi_1,\chi_2)))^{\vee}$.
\end{proof}

\begin{remark}\label{restrict_to_GQp_PS}
	Let $L=\bQ_{p^2}(\sqrt{p})$ be the abelian extension $\bQ_p$ with norm group $S$. By equation (\ref{res_to_S}), the underlying Galois representation of $\wt{\mathbf{V}}(\tilde{\pi}(\chi_1,\chi_2))$ is given by
	\[
	\mathbf{V}(\tilde{\pi}(\chi_1,\chi_2))\cong \Ind^{\mathscr{G}_{\bQ_p}}_{\mathscr{G}_L}(\chi_2|_{\mathscr{G}_L}),
	\]
	which is the direct sum over all characters of $\mathscr{G}_{\bQ_p}$ extending $\chi_2|_{\mathscr{G}_L}$.
\end{remark}

\subsubsection{Genuine supersingular} For computing the image of an absolutely irreducible genuine supersingular representation, we make use of Proposition \ref{class_irred}, which allows us to write such a representation as $\tilde{\pi}(r,0,\eta)$ for some $0\leq r\leq p-1$ with $r\neq \frac{p-1}{2}$ and a smooth character $\eta\colon \bQ_p^{\times}\to k^{\times}$. Since the functor $\wt{\mathbf{D}}$ is compatible with twist, Remark \ref{Dtilde_twist}, we may assume that $\eta$ is the trivial character, i.e.\ we only need to treat the representation $\tilde{\pi}(r,0)$.

Let now $I=\begin{mat}
	\bZ_p^{\times} & \bZ_p\\ p\bZ_p & \bZ_p^{\times}
\end{mat}$ be the Iwahori subgroup. It admits the semi-direct product decomposition $I=I_1\rtimes H$, where $I_1=\begin{mat}
	1+p\bZ_p & \bZ_p\\p\bZ_p & 1+p\bZ_p
\end{mat}$ is the pro-$p$ Iwahori subgroup and $H$ is the torus of $\GL_2(\bF_p)$ viewed as a subgroup of $I$ via the Teichmüller lift. Since the order of the finite abelian group $H$ is prime to $p$, every finite dimensional $k$-representation of $H$ decomposes as the direct sum of characters. Given a character $\chi\colon H\to \bF_p^{\times}$, we let $e_{\chi}\in \bF_p[H]$ denote the idempotent cutting out the $\chi$-isotypic subspace. 

\begin{definition}[{\cite[Definition 3.8]{witthaus_meta}}] For a character $\chi\colon H\to \bF_p^{\times}$ and $i,j\in \bZ/2\bZ$, define $\chi[i,j]=\chi\otimes (\omega^{i\frac{p-1}{2}}\otimes \omega^{j\frac{p-1}{2}})\colon H\to \bF_p^{\times}$.
\end{definition}

Given a character $\chi=\chi_1\otimes \chi_2$ of $H$, we put $\chi^s=\chi_2\otimes \chi_1$. Note that the element $\tilde{\Pi}=(\begin{mat}
	0 & 1\\ p & 0
\end{mat},1)\in \wt{G}$ normalizes $I_1$ and thus acts on the space of $I_1$-invariants $\pi^{I_1}$ of any $\wt{G}$-representation $\pi$. In fact, if $v\in \pi^{I_1}e_{\chi}$, then $\tilde{\Pi}v\in \pi^{I_1}e_{\chi^s[1,0]}$.\\

By definition, $\tilde{\pi}(r,0)$ is a quotient of $\cInd^{\wt{G}}_{\wt{K}Z(\wt{G})}(\Sym^r(k^2)\boxtimes \iota)$ and thus contains the weight $\Sym^r(k^2)$, whose $I_1$-invariants are one-dimensional, equal to the character $\chi:=\omega^r\otimes \mathbf{1}\colon H\to k^{\times}$. Let us fix a non-zero element $v_1\in \Sym^r(k^2)^{I_1}\subset \tilde{\pi}(r,0)^{I_1}$, and define $v_i=(\tilde{\Pi})^{i-1}v_1\in \tilde{\pi}(r,0)^{I_1}$ for $1\leq i\leq 4$. By \cite[Corollary 4.52]{witthaus_meta}, we then have $\tilde{\pi}(r,0)^{I_1}=\oplus_{i=1}^4 kv_i$. If we let $\chi_i\colon H\to k^{\times}$ denote the $H$-eigencharacter of $v_i$, then
\begin{equation}\label{H-eigencharacters}
	\chi_1=\chi, \hspace{0.2cm}\chi_2=\chi^s[1,0], \hspace{0.2cm} \chi_3=\chi[1,1], \hspace{0.2cm}  \chi_4=\chi^s[0,1].
\end{equation}
Moreover, Corollary 4.46 in \textit{loc.\ cit.} tells us that each $v_i$ generates a (smooth) irreducible $k$-representation of $K$, which we can then write as $\Sym^{r_i}(k^2)\otimes {\det}^{b_i}$ for a unique $0\leq r_i\leq p-1$ and $b_i\in \bZ/(p-1)\bZ$. In fact, if we put
\[
r'=\begin{cases}
	\frac{p-1}{2}-r \text{ if $0\leq r<\frac{p-1}{2}$}\\
	\frac{3(p-1)}{2}-r \text{ if $\frac{p-1}{2}<r\leq p-1$,}
\end{cases}
\]
then
\[
(r_1,b_1)=(r,0), \hspace{0.2cm} (r_2,b_2)=(r',r), \hspace{0.2cm} (r_3,b_3)=(r,\frac{p-1}{2}), \hspace{0.2cm} (r_4,b_4)=(r',r+\frac{p-1}{2}).
\]
Finally, let us define $s_i=r_{i+1}$ for $1\leq i\leq 4$ and
\[
c_1=(-1)^r (r')!, \hspace{0.2cm} c_2=(-1)^{\frac{p-1}{2}} r!, \hspace{0.2cm} c_3=(-1)^{r+\frac{p-1}{2}}(r')!, \hspace{0.2cm} c_4=(-1)^{\frac{p-1}{2}} r!.
\]

\begin{lem}[{\cite[Corollary 4.48]{witthaus_meta}}]\label{check_assumption}
	For each $1\leq i\leq 4$, $X^{s_i}F(v_i)=c_iv_{i+1}$.
\end{lem}

Recall that $(-)|_S\colon \mphigamma\to \Mod^{\sm}_S(\phigamma)$ denotes the restriction functor. In the language of Galois representations, we denote this more accurately by $(-)|_{S\times \GQp}\colon{\Mod}^{\fin}_{\GQp,\iota}(k)\to \Mod^{\fin}_{S\times \GQp}(k)$.

\begin{prop}\label{image_ss}
	\begin{enumerate}
		\item[{\rm (i)}] The underlying Galois representation $\mathbf{V}(\tilde{\pi}(r,0))$ of the metaplectic Galois representation $\wt{\mathbf{V}}(\tilde{\pi}(r,0))$ is irreducible.
		\item[{\rm (ii)}] The restriction to $S\times \GQp$ is given by
		\[
		\wt{\mathbf{V}}(\tilde{\pi}(r,0))|_{S\times \GQp} \cong \omega^r|_S\boxtimes  \omega^{r-1}\mu_{\lambda} \otimes \Ind^{\GQp}_{\mathscr{G}_{\bQ_{p^4}}}(\omega_4^{\frac{p^2+1}{2}s}),
		\]
		where $\lambda^4=(-1)^{\frac{p-1}{2}}(r!)^2(r'!)^2$ and 
		\[
		s=\begin{cases}
			-2r+p \text{ if $r<\frac{p-1}{2}$}\\
			-2r+3p \text{ if $r>\frac{p-1}{2}$}.
		\end{cases}
		\]
		\item[{\rm (iii)}] The canonical $\wt{B}$-equivariant map $\wt{\bOmega}\wt{\mathbf{D}}(\tilde{\pi}(r,0))\xrightarrow{\cong} \tilde{\pi}(r,0)$ of Proposition \ref{can_B_equiv_map} is an isomorphism.
		
		\item[{\rm (iv)}] The metaplectic $(\varphi,\Gamma)$-module $\wt{\mathbf{D}}(\tilde{\pi}(r,0))$ is the unique object in $\mphigamma$ satisfying (ii) (after passing to Galois representations) such that the $\wt{B}$-action on $\wt{\bOmega}\wt{\mathbf{D}}(\tilde{\pi}(r,0))$ extends to a smooth $\wt{G}$-action.
	\end{enumerate}
\end{prop}

\begin{proof}
	First, note that the center $Z(\wt{G})\cong S\times \mu_2$ acts on $\tilde{\pi}(r,0)$ via the genuine character $\omega^r|_S\boxtimes \iota$, which trivially implies that the action of $S$ on $\wt{\mathbf{V}}(\tilde{\pi}(r,0))$ is given by $\omega^r|_S$, so it is enough to determine the underlying Galois representation. The previous discussion places us in the setup of Section \ref{Galois_calculate} with $M=\tilde{\pi}(r,0)^{I_1}$. By \cite[Proposition 4.49]{witthaus_meta}, we have $\tilde{\pi}(r,0)=\pi_M[F^{-1}]$, where $\pi_M$ is the $P^+$-subrepresentation generated by $M$. By Lemma \ref{basic_lemma} (i) and (ii), the subspace of $X$-torsion elements in $\pi_M$ is equal to $M$ (which is finite dimensional), hence $\pi_M\in \mathcal{M}(\pi)$. We may therefore apply Proposition \ref{stabilize} to deduce that $\mathbf{D}(\tilde{\pi}(r,0))=(\pi_M^{\vee}[1/X])'$. Corollary \ref{Galois_calculate}, which we may apply by Lemma \ref{check_assumption}, now proves part (i) and (ii). The proof of it also shows that $\pi_M^{\vee}=(\pi_M^{\vee}[1/X])^{\natural}$, which is equal to $(\pi_M^{\vee}[1/X])^{\sharp}$ by Example \ref{D_nat=D_sharp} (ii). The canonical map $\wt{\bOmega}\wt{\mathbf{D}}(\tilde{\pi}(r,0))\to \tilde{\pi}(r,0)$ being an isomorphism is now just a reformulation of the equality $\tilde{\pi}(r,0)=\pi_M[F^{-1}]=\varinjlim_F \pi_M$.
	
	For part (iv), write $\wt{\mathbf{D}}(\tilde{\pi}(r,0))=(D,\theta)$. Suppose that $\lambda\colon \wt{\bQ_p^{\times}}\to \Aut_{k((X))}(D)$ is another action such that $(D,\lambda)$ is a metaplectic $(\varphi,\Gamma)$-module satisfying $\lambda|_{\wt{S}}=\theta|_{\wt{S}}$. Then the $(\varphi,\Gamma)$-equivariant action of $\wt{\bQ_p^{\times}}$ on $D$ defined by $\lambda\theta^{-1}$ factors through $\wt{S}\setminus \wt{\bQ_p^{\times}}\cong \bZ/2\bZ \times \bZ/2\bZ$. Since $D$ is absolutely irreducible (e.g.\ by part (ii)), Schur's Lemma implies that this action is given by scalars. In conclusion, $\lambda=\theta \otimes \varepsilon$ for some $\varepsilon=\mu_{-1}^{a}\omega^{b\frac{p-1}{2}}$ for $a,b\in \bZ/2\bZ$. In particular, there is a $\wt{B}$-equivariant isomorphism
	\[
	\wt{\bOmega}(D,\lambda)\cong \wt{\bOmega}(D,\theta)\otimes (\mathbf{1}\otimes \varepsilon),
	\]
	where the character $\mathbf{1}\otimes \varepsilon$ of the torus $T$ is viewed as a character of $\wt{B}$ via inflation. By part (iii), we are reduced to showing that the $\wt{B}$-action on $\tilde{\pi}(r,0)\otimes (\mathbf{1}\otimes \varepsilon)$ extends to a smooth $\wt{G}$-action if and only if $\varepsilon$ is trivial. Assume that the action does extend. It must then necessarily define a genuine supersingular representation (as its restriction to $\wt{B}$ is irreducible). By \cite[Corollary 4.55]{witthaus_meta}, the isomorphism class of a genuine supersingular representation is determined by the restriction to $P$. In particular, we obtain a $\wt{B}$-equivariant (in fact, $\wt{G}$-equivariant) isomorphism $f\colon \tilde{\pi}(r,0)\cong \tilde{\pi}(r,0)\otimes (\mathbf{1}\otimes \varepsilon)$. We first check that $\varepsilon$ is unramified. The argument at the beginning of the proof of the corollary just used goes through to show that $f$ restricts to a $\Gamma$-equivariant isomorphism $M\cong M\otimes (\mathbf{1}\otimes \varepsilon)$. We repeat it for the convenience of the reader: Since $\tilde{\pi}(r,0)=\pi_M[F^{-1}]$, we may choose $m\gg 0$ with $F^{4m}f(M)\subset \pi_M\otimes (\mathbf{1}\otimes \varepsilon)$. An iteration of the relations in Lemma \ref{check_assumption} implies that $X^{e(i)_m}F^{4m}v_i=c^mv_i$ for $e(i)_m$ and $c$ as defined prior to (\ref{iso_pi_{i,m}}), and so $f(M)\subset \pi_M\otimes (\mathbf{1}\otimes \varepsilon)$. Comparing $X$-torsion shows that $f(M)\subset M\otimes (\mathbf{1}\otimes \varepsilon)$, which must be an equality for dimension reasons. Recall that $\chi_i$ denotes the $H$-eigencharacter of $v_i$. The $H$-equivariance of $f$ forces $f(v_i)$ to lie in the subspace $(M\otimes (\mathbf{1}\otimes \varepsilon))e_{\chi_i}=Me_{\chi_i[0,b]}$. The explicit description of the characters (\ref{H-eigencharacters}) implies that, if $b\neq 0$, then this space can only possibly be non-zero if $\chi_i[0,b]=\chi_j$ with $j-i\equiv 1 \bmod 2$, in which case $s_i\neq s_j$. Write $f(v_i)=av_j$ for some $a\in k^{\times}$. By symmetry, we may assume that $s_i>s_j$. Applying $X^{s_i}F$ and utilizing Lemma \ref{check_assumption} once again, gives the contradiction
	\[
	0\neq c_if(v_{i+1})=f(X^{s_i}F(v_i))=aX^{s_i}F(v_j)=ac_jX^{s_i-s_j}v_{j+1}=0
	\]
	since $M$ is killed by $X$. This proves that $b=0$, i.e.\ $\varepsilon$ is unramified, and we must have $f(v_i)=a_iv_i$ for some $a_i\in k^{\times}$ for all $1\leq i\leq 4$. Applying $f$ to the relation $X^{s_i}F(v_i)=c_iv_{i+1}$ forces $a_i=a_{i+1}$ for all $1\leq i\leq 4$. However, the $\tilde{\Pi}^2$-equivariance of $f$ implies that $a_3=\varepsilon(p)a_1$ and thus $\varepsilon(p)=1$. In conclusion, $\varepsilon$ is the trivial character.
\end{proof}

By definition, the underlying Galois representation of a metaplectic one is invariant under twist by any character of order two. The unramified part of such a character is trivial when restricted to $\mathscr{G}_{\bQ_{p^2}}$ (hence also when restricted to $\mathscr{G}_{\bQ_{p^4}}$), so the interesting part left when considering absolutely irreducible representation of dimension four is invariance under twist by the character $\omega^{\frac{p-1}{2}}$.

\begin{lem}\label{galois_lemma1}
	Up to twist by a character, the absolutely irreducible Galois representations of $\GQp$ on $4$-dimensional $k$-vector spaces invariant under twist by $\omega^{\frac{p-1}{2}}$ are precisely those of the form $\Ind^{\mathscr{G}_{\bQ_p}}_{\mathscr{G}_{\bQ_{p^4}}}(\omega_4^{\frac{p^2+1}{2}h})$ for some odd integer $h\in \bZ$ with $1\leq h\leq 2(p^2-1)$
\end{lem}

\begin{proof}
	We first show that the representations of the form $\Ind^{G_{\bQ_p}}_{G_{\bQ_{p^4}}}(\omega_4^{\frac{p^2+1}{2}h})$, with $1\leq h\leq 2(p^2-1)$ odd, are irreducible and invariant under twist by $\omega^{\frac{p-1}{2}}$. For irreducibility, it suffices to show, by Proposition \ref{irred_Galois}, that $\frac{p^2+1}{2}h$ is primitive relative to $4$, i.e.\ that $\frac{p^2+1}{2}h$ is not divisible by $\frac{p^4-1}{p^{d}-1}$ for $d=1,2$. This is clear since $h$ is odd. The invariance under twist by $\omega^{\frac{p-1}{2}}$ is equivalent to $\omega^{\frac{p-1}{2}}$ being equal to $\omega_4^{(p^{i}-1)\frac{p^2+1}{2}h}$ for some $0\leq i\leq 3$. Noting that $\omega^{\frac{p-1}{2}}=\omega_4^{\frac{p^4-1}{2}}$, it is enough to observe that, since $h$ is odd, we have 
	\[
	\frac{p^4-1}{2}\equiv \frac{p^4-1}{2}h=(p^2-1)\frac{p^2+1}{2}h \bmod p^4-1.
	\]
	Assume now that $\rho$ is an arbitrary $4$-dimensional absolutely irreducible Galois representation invariant under twist by $\omega^{\frac{p-1}{2}}$. By the proposition just used, we may, after possibly twisting, write $\rho\cong \Ind^{\mathscr{G}_{\bQ_p}}_{\mathscr{G}_{\bQ_{p^4}}}(\omega_4^s)$ for some $1\leq s\leq p^4-1$ primitive relative to $4$. Invariance under twist by $\omega^{\frac{p-1}{2}}$ gives $\frac{p^4-1}{2}\equiv (p^{i}-1)s \bmod p^4-1$ for some $0\leq i\leq 3$. The case $i=0$ is excluded for trivial reasons. If $i=3$, then multiplying by the odd prime $p$ gives
	\[
	\frac{p^4-1}{2}\equiv p\frac{p^4-1}{2}\equiv (p^4-p)s\equiv -(p-1)s \bmod p^4-1,
	\]
	or equivalently $\frac{p^4-1}{2}\equiv (p-1)s \bmod p^4-1$, which is the case $i=1$. In the cases $i=1,2$, we obtain an equality of the form $s=(2a+1)\frac{p^2-1}{p^{i}-1}\frac{p^2+1}{2}$, for some $a\in \bZ$. If $i=1$, then this implies that $\frac{p^4-1}{p^2-1}=p^2+1$ divides $s$, contradicting the irreducibility of $\rho$ (or equivalently primitivity of $s$). In case $i=2$, we obtain the desired result.
\end{proof}

\begin{lem}\label{galois_lemma2}
	Given an odd integer $h\in \bZ$, there exist integers $a,h'\in \bZ$ with $3\leq h'\leq 2p-1$ odd such that
	\[
	\Ind^{\mathscr{G}_{\bQ_p}}_{\mathscr{G}_{\bQ_{p^4}}}(\omega_4^{\frac{p^2+1}{2}h})\cong \omega^{a}\otimes \Ind^{\mathscr{G}_{\bQ_p}}_{\mathscr{G}_{\bQ_{p^4}}}(\omega_4^{\frac{p^2+1}{2}h'}).
	\]
\end{lem}

\begin{proof}
	We first show that we can choose $1\leq h'\leq 2p-1$ (i.e.\ including $h'=1$) appropriately.
	Since $\omega_4$ has order $p^4-1$, we may without loss of generality assume that $0\leq h < 2(p^2-1)$. So, if $h=h_0+ph_1+p^2h_2$ denotes the base-$p$ expansion with $0\leq h_i\leq p-1$, then $h_2\in \{0,1\}$. Next, note that replacing $h$ by $h-2a(p+1)$ for some $a\in \bZ$ amounts to twisting the induced representation by $\omega^{-a}$. If $h_2=1$, we take $a=\frac{p-1}{2}$ and may therefore assume that $0\leq h < p^2-1$, whence $h_2=0$. Choose $a\in \bZ$ such that $h_1':=h_1-2a\in \{0,1\}$ and put $h_0':=h_0-2a$. We may replace $h$ by $h'=h-2a(p+1)=h_0'+ph_1'$. Note that $|h_0'|\leq p-1$. If $h_0'\geq 0$ or $h_1'=1$, then $1\leq h'\leq 2p-1$ as desired. Assume now that $h_1'=0$ and $h'_0<0$. Replace $h'=h'_0$ by $h'_0+2(p+1)=2p+(2+h'_0)$. If $-h'_0\geq 3$, then we are done; otherwise $h'_0=-1$ (as it must be odd), in which case we use the intertwining
	\[
	\Ind^{\mathscr{G}_{\bQ_p}}_{\mathscr{G}_{\bQ_{p^4}}}(\omega_4^{\frac{p^2+1}{2}(2p+1)})\cong 	\Ind^{\mathscr{G}_{\bQ_p}}_{\mathscr{G}_{\bQ_{p^4}}}(\omega_4^{p\frac{p^2+1}{2}(2p+1)}) =	\Ind^{\mathscr{G}_{\bQ_p}}_{\mathscr{G}_{\bQ_{p^4}}}(\omega_4^{\frac{p^2+1}{2}(p+2)}).
	\]
	This proves the statement with $1\leq h'\leq 2p-1$ odd. Similar to the first isomorphism in the above intertwining, we may replace $h'$ by $ph'$ and so replace $h'=1$ by $p$. In particular, we may assume $3\leq h'\leq 2p-1$.
\end{proof}

\begin{thm}\label{ss_bij_galois}
	The functor $\mathbf{V}\colon \Mod^{\fadm}_{\wt{G}}(k)\to \Mod^{\fin}_{\mathscr{G}_{\bQ_p}}(k)$ induces a bijection
	\[
	\left\{\begin{array}{c}
		\text{Absolutely irreducible}\\
		\text{genuine supersingular}\\
		\text{$k$-representations of $\wt{G}$}
	\end{array}\right\} \cong \left\{\begin{array}{c}
		\text{Absolutely irreducible}\\
		\text{smooth $\rho\colon \GQp\to \GL_4(k)$}\\
		\text{s.t.\ $\rho\cong \rho\otimes \omega^{\frac{p-1}{2}}$}
	\end{array}\right\},
	\]
	both sides considered up to isomorphism.
\end{thm}

\begin{proof}
	This follows from Proposition \ref{image_ss}, Lemma \ref{galois_lemma1}, Lemma \ref{galois_lemma2} and the fact, \cite[Corollary 4.55]{witthaus_meta}, that the isomorphism class of a supersingular representation is determined by its restriction to the mirabolic subgroup $P$. Here, note that the restriction of the canonical map of part (iii) of the cited proposition to the mirabolic subgroup does not use the metaplectic structure.
\end{proof}

\begin{remark}
	Fix a square root $\varpi=\sqrt{-p}$, which serves as a uniformizer in the abelian extension $L=\bQ_{p^2}(\sqrt{p})$ of $\bQ_p$ with norm group $S$. Consider the character
	\[
	\omega_L \colon \mathscr{G}_L \twoheadrightarrow \bF_{p^2}^{\times}, \hspace{0.2cm} g(\sqrt[p^2-1]{\varpi})=\omega_L(g) \sqrt[p^2-1]{\varpi}, \text{ for all $g\in G_L$};
	\]
	it is independent of the chosen $p^2-1$-th root of $\varpi$ since $\mu_{p^2-1}(\ol{\bQ}_p)\subset \bQ_{p^2}\subset L$. We then have $\omega_4^{(p^2+1)/2}|_{\mathscr{G}_{L}\cap \mathscr{G}_{\bQ_{p^4}}}=\omega_L|_{\mathscr{G}_{L}\cap \mathscr{G}_{\bQ_{p^4}}}$. In particular, for any odd integer $h\in \bZ$, we have an isomorphism
	\[
	\Ind^{\mathscr{G}_{\bQ_p}}_{\mathscr{G}_{\bQ_{p^4}}}(\omega_4^{\frac{p^2+1}{2}h})\cong \Ind^{\mathscr{G}_{\bQ_p}}_{\mathscr{G}_L}(\omega_L^h)
	\]
	of irreducible $\mathscr{G}_{\bQ_p}$-representations. Conversely, the character $\omega_L^2$ extends to $\mathscr{G}_{\bQ_{p^2}}$, e.g.\ $\omega_2$, so the right-hand side will be reducible if $h$ is even.
\end{remark}

\bibliography{metaplectic_montreal_arxiv}
\bibliographystyle{plain}

\end{document}